\pdfoutput=1
\documentclass[11pt,letterpaper]{article}
\usepackage[letterpaper,margin=1in]{geometry}
\usepackage[utf8]{inputenc}
\usepackage[T1]{fontenc}
\usepackage[dvipsnames]{xcolor}
\usepackage{amsmath,amssymb,amsthm,mathtools,mathrsfs}
\usepackage{enumitem,graphicx,float}
\graphicspath{{v2/}}
\usepackage{tikz}
\usetikzlibrary{arrows.meta,calc,positioning,decorations.pathreplacing}
\usepackage{microtype}
\microtypesetup{expansion=false}
\usepackage[numbers,square]{natbib}
\setcitestyle{numbers,open={[},close={]},citesep={,}}
\usepackage{mathptmx}
\usepackage{courier}
\usepackage{aliascnt}
\usepackage[unicode=true,colorlinks=true,breaklinks=true]{hyperref}
\definecolor{selectedpurple}{HTML}{79538A}
\definecolor{curvepurple}{HTML}{6F3F8F}
\hypersetup{linkcolor=selectedpurple!85!black,citecolor=YellowOrange!85!black,
  urlcolor=Aquamarine!85!black,
  pdftitle={Silver Rate Is (Almost) Optimal for Gradient Descent Acceleration},pdfauthor={Yuhan Ye and Kaizhao Liu}}
\newtheorem{theorem}{Theorem}[section]
\newaliascnt{lemma}{theorem}
\newtheorem{lemma}[lemma]{Lemma}
\aliascntresetthe{lemma}
\newaliascnt{proposition}{theorem}

\aliascntresetthe{proposition}
\newaliascnt{corollary}{theorem}

\aliascntresetthe{corollary}
\newaliascnt{fact}{theorem}
\newtheorem{fact}[fact]{Fact}
\aliascntresetthe{fact}
\theoremstyle{definition}
\newaliascnt{definition}{theorem}

\aliascntresetthe{definition}
\newaliascnt{remark}{theorem}

\aliascntresetthe{remark}
\usepackage[nameinlink,capitalize]{cleveref}
\crefname{fact}{Fact}{Facts}
\Crefname{fact}{Fact}{Facts}
\crefname{appendix}{Appendix}{Appendices}
\Crefname{appendix}{Appendix}{Appendices}
\allowdisplaybreaks
\newcommand{\R}{\mathbb R}
\newcommand{\N}{\mathbb N}
\newcommand{\psil}{p_{\mathrm{sil}}}
\newcommand{\pany}{p_{\mathrm{any}}}

\DeclareMathOperator*{\argminop}{arg\,min}

\title{Silver Rate Is (Almost) Optimal for Gradient Descent}
\author{Yuhan Ye\footnotemark[1]\\MIT\\\texttt{yyh03@mit.edu}
  \and Kaizhao Liu\thanks{Authors are listed in random order.}\\MIT\\\texttt{mrzt@mit.edu}}
\date{September 8, 2026}

\begin{document}
\maketitle
\begin{abstract}
Gradient descent (GD) is one of the most fundamental algorithms in optimization.
Yet even for smooth convex objectives, the optimal rate among predetermined nonnegative stepsize schedules has remained unresolved.
Altschuler and Parrilo~\citep{AltschulerParrilo2025} showed that their recursive \emph{silver stepsize schedule} achieves an $O(n^{-\psil})$ rate, where $\psil=\log_2(1+\sqrt2)\approx1.2716$, improving upon the classical $O(n^{-1})$ rate of constant-step GD.
Although the optimality of $\psil$ has been well supported by numerical evidence for small $n$, a matching lower bound has remained missing.

In this paper, we prove an $\Omega(n^{-\psil-o(1)})$ lower bound for GD, showing that $\psil$ is the optimal convergence exponent. Additionally, we prove an \emph{anytime} lower bound that matches the exponent $2\psil/(1+\psil)\approx1.1195$ of the known upper bound~\citep{ZhangLeeDuChen2025}.

Building on the approaches of~\citep{MaChen2026,JungChoYun2026}, we select specific iterations as \emph{checkpoints} at which the GD trajectory explores a new coordinate, and we couple these coordinates using a chain of local functions.
Our main technical contribution is a local hard function that \emph{bends} the GD trajectory between successive checkpoints, resulting in a sharpened lower bound.
For each prescribed stepsize schedule, we maximize this bound over admissible checkpoint sets.
A recursive analysis of this maximum then yields a lower bound matching the silver exponent for every such schedule.
\end{abstract}

\clearpage
\tableofcontents
\clearpage

\section{Introduction}
\label{sec:introduction}

We consider the minimization of a smooth convex function $f:\R^d\to\R$.
A standard algorithm for this problem is \textbf{gradient descent} (GD), with a history dating back to Cauchy~\citep{Cauchy1847}.
Given a \emph{stepsize schedule} $H=(h_1,\ldots,h_n)\in[0,\infty)^n$ fixed before the optimization begins, GD proceeds via the iteration
\[
  x_t=x_{t-1}-h_t\nabla f(x_{t-1}),\qquad 1\le t\le n.
\]
In this paper, we study the limits of this method's convergence rate.
To formalize this question, let $\mathcal F_L(\R^d)$ be the class of convex $L$-smooth functions on $\R^d$ with a nonempty set of minimizers.
We measure the convergence rate achieved by a schedule $H$ using its worst-case (normalized) function suboptimality
\begin{equation}
  \mathcal R_n(H):=
  \sup_{d\in\N}\;
  \sup_{f\in\mathcal F_L(\R^d)}\;
  \sup_{x_*\in\argminop f}\;
  \sup_{x_0\in\R^d\setminus\{x_*\}}
  \frac{f(x_n)-f(x_*)}{\frac{L}{2}\lVert x_0-x_*\rVert^2}.
  \label{eq:worst-case-error}
\end{equation}

\paragraph{Classical Results.}
Classical convergence analyses of gradient methods can be traced to~\citep{LevitinPolyak1966}.
With the constant stepsize of $1/L$, GD achieves the classical $O(n^{-1})$ rate; see the textbooks~\citep{Polyak1987,Nesterov2004,Bubeck2015}.
This rate is tight, even in one dimension~\citep{DroriTeboulle2014}.

To obtain faster rates, classical acceleration methods introduce momentum or auxiliary sequences, as in Polyak's heavy-ball method~\citep{Polyak1964} and Nesterov's accelerated method~\citep{Nesterov1983}.
The latter attains an $O(n^{-2})$ rate for smooth convex objectives, matching the classical $\Omega(n^{-2})$ lower bound for deterministic first-order methods established by Nemirovsky and Yudin~\citep{NemirovskyYudin1979}.\footnote{\label{fn:first-order-models}A \emph{deterministic first-order method} accesses the objective only through function values and gradients at arbitrary points. Nesterov's simpler quadratic-chain proof assumes $x_t\in x_0+\operatorname{span}\{\nabla f(x_0),\ldots,\nabla f(x_{t-1})\}$, referred to as the \emph{linear-span assumption}~\citep[Assumption~2.1.4]{Nesterov2004}.}
These accelerated methods, however, modify the original GD update.
For GD with predetermined stepsizes, the classical results therefore leave a gap between the $O(n^{-1})$ upper bound and the $\Omega(n^{-2})$ lower bound.

\paragraph{\texorpdfstring{Acceleration via the Silver Stepsize Schedule.}{Acceleration via the Silver Stepsize Schedule.}}
For many years, the $O(n^{-1})$ rate remained the best general upper bound known for GD, leaving open the question of whether stepsize choice alone could improve it.
In recent years, computer-assisted analyses have opened up new opportunities for acceleration by helping identify better stepsize schedules.
Drori and Teboulle~\citep{DroriTeboulle2014} introduced \emph{performance estimation problems} (PEPs), which express a method's worst-case error after a prescribed number of iterations as an optimization problem.
\emph{Exact interpolation conditions} characterize when prescribed function values and gradients at finitely many points come from a single function in the specified class~\citep{TaylorHendrickxGlineur2017}.
These conditions give exact semidefinite formulations of PEPs for \emph{fixed-step methods}.\footnote{Fixed-step methods use predetermined coefficients to combine past gradients and may use momentum.}
Using PEPs to analyze and optimize GD stepsize schedules over short horizons, Altschuler showed in his thesis~\citep[Chapter~8]{Altschuler2018} that coordinating short and long steps can improve convergence in the strongly convex setting.
Such gains require a joint analysis of multiple steps, since separate bounds on individual steps fail to reflect this advantage.
Later numerical experiments also suggested that a rate faster than $O(n^{-1})$ was possible for general smooth convex objectives~\citep{DasGuptaVanParysRyu2024}.

In a major breakthrough, Altschuler and Parrilo~\citep{AltschulerParrilo2025HedgingI} introduced the recursive \emph{silver stepsize schedule}.
On smooth strongly convex objectives, this reduces the squared distance to the minimizer by a factor of $\varepsilon$ in $O(\kappa^{1/\psil}\log(1/\varepsilon))$ iterations, where $\kappa$ is the condition number and $\psil:=\log_2(1+\sqrt2)\approx1.2716$.\footnote{$1+\sqrt2$ is the silver ratio, whose base-$2$ logarithm is $\psil$. For smooth strongly convex objectives, constant-step GD has an $O(\kappa\log(1/\varepsilon))$ iteration bound, while Nesterov acceleration achieves $O(\sqrt\kappa\log(1/\varepsilon))$ using momentum~\citep{Nesterov1983}.}
They later extended their schedule to smooth convex objectives, obtaining an $O(n^{-\psil})$ rate~\citep{AltschulerParrilo2025}.
In the same work, they conjectured that $\psil$ is the optimal polynomial convergence exponent.
Subsequent work improved the constants in the bound on function suboptimality and established accelerated gradient-norm guarantees~\citep{GrimmerShuWang2025LongSteps}.
Further work also gave rules for combining shorter schedules to obtain the silver rate at arbitrary prescribed horizons~\citep{GrimmerShuWang2025Composing,ZhangJiang2026}.
For a broader introduction to stepsize-based acceleration, see the recent survey~\citep{AltschulerParrilo2026Survey}.

The smooth convex silver-rate results above concern the \emph{non-anytime (finite-horizon) setting}, in which the schedule may be chosen separately for each fixed horizon $n$.
In the \emph{anytime setting}, one infinite schedule $H=(h_t)_{t\ge1}$ is used across all horizons, with $H_n=(h_1,\ldots,h_n)$.
By combining silver-schedule blocks,~\citep{ZhangLeeDuChen2025} construct such a schedule with an $O(n^{-\pany})$ guarantee at every stopping time, where $\pany:=2\psil/(1+\psil)\approx1.1195$.
This resolves the open problem~\citep{KornowskiShamir2024} of whether stepsize choice alone can accelerate GD uniformly over all stopping times.

\paragraph{Recent Progress on Lower Bounds.}
Despite these improved upper bounds, the best-known lower bound for GD until recently was the classical $\Omega(n^{-2})$ bound~\citep{NemirovskyYudin1979}.
Nesterov~\citep{Nesterov2004} provides a simple \emph{quadratic zero-chain} construction for this lower bound.
By initializing at the origin, a method satisfying the linear-span assumption can reach at most one new coordinate per iteration, while the minimizer has nonzero entries farther along the chain.
Even exact minimization over the reached coordinates leaves an $\Omega(n^{-2})$ error.
Using PEP,~\citep{KimFessler2016} developed the optimized gradient method (OGM)\footnote{OGM adds momentum terms to gradient steps, so it is not a pure GD method.}, whose bound on function suboptimality was later proved optimal in both rate and constant~\citep{Drori2017}.
The lower bound construction in~\citep{Drori2017} combines smooth convex interpolation with a zero-chain property.
These results sharpened the constants of the classical first-order lower bound, but did not improve the convergence exponent for GD with predetermined stepsizes.

In recent months, a sequence of results brought the lower bounds closer to the silver-rate upper bounds $O(n^{-\psil})$ and $O(n^{-\pany})$ in the non-anytime and anytime settings, respectively.
In the first breakthrough beyond the classical $\Omega(n^{-2})$ bound, \citep{TsaiFatkhullinZhangHe2026} proved an $\Omega(n^{-4/3})$ anytime lower bound.\footnote{Throughout, we use the phrase `$\Omega(n^{-p})$ anytime lower bound' as shorthand for the statement that no single infinite stepsize schedule achieves an $o(n^{-p})$ rate.
The classical lower bound and Nesterov's matching upper bound also hold in the anytime setting.}
\citep{MaChen2026} then made an important advance in the non-anytime setting by proving an $\Omega(n^{-1.932})$ lower bound.
A subsequent blog post~\citep{Tsai2026} sharpened it to $\Omega(n^{-\sqrt3})$.
In our previous work~\citep{YeLiu2026}, we proved an $\Omega(n^{-1.6342})$ non-anytime lower bound and an $\Omega(n^{-1.2408})$ anytime lower bound for arbitrary real-valued stepsizes.
To the best of our knowledge, these remain the strongest known lower bounds when negative stepsizes are allowed.
For nonnegative schedules, \citep{JungChoYun2026} further improved the non-anytime and anytime lower bounds to $\Omega(n^{-1.4500})$ and $\Omega(n^{-1.1837})$, respectively.
Their hard instances form a chain of local one-sided Huber functions\footnote{The one-sided Huber function is defined in \cref{sec:classical-constructions}.} coupling adjacent coordinates, much like Nesterov's quadratic chain.

\subsection{Main Results}

A key limitation of the local one-sided Huber functions in \citep{JungChoYun2026} is that their fixed gradient directions restrict the tightness of the resulting lower bounds. We address this by refining these local functions to \emph{bend the gradient descent trajectory}, which allows us to establish tighter lower bounds with exponents $\psil+o(1)$ and $\pany+o(1)$.

\begin{theorem}[Finite-horizon error lower bound]
\label{thm:nonanytime}
There is an absolute constant $C>0$ such that for every sufficiently large $n$ and every schedule $H\in[0,\infty)^n$,
\begin{equation}
  \mathcal R_n(H)\ge
  n^{-\left(\psil+C\sqrt{\frac{\log\log n}{\log n}}\right)}.
  \label{eq:main-nonanytime}
\end{equation}
\end{theorem}

\begin{theorem}[Anytime error lower bound]
\label{thm:anytime}
There is an absolute constant $C>0$ such that every infinite nonnegative schedule $H=(h_t)_{t\ge1}$ satisfies
\begin{equation}
  \mathcal R_n(H_n)\ge
  n^{-\left(\pany+C\sqrt{\frac{\log\log n}{\log n}}\right)}
  \label{eq:main-anytime}
\end{equation}
for infinitely many $n$.
\end{theorem}

Together with the corresponding upper bounds~\citep{AltschulerParrilo2025,ZhangLeeDuChen2025}, our results identify $\psil$ and $\pany$ as the optimal polynomial convergence exponents in the non-anytime and anytime settings, respectively.
This resolves the conjecture of Altschuler and Parrilo~\citep{AltschulerParrilo2025} that the silver exponent is optimal. \Cref{sec:overview} explains the idea behind our hard function construction.

\begin{figure}[H]
  \centering
  \includegraphics[width=\textwidth]{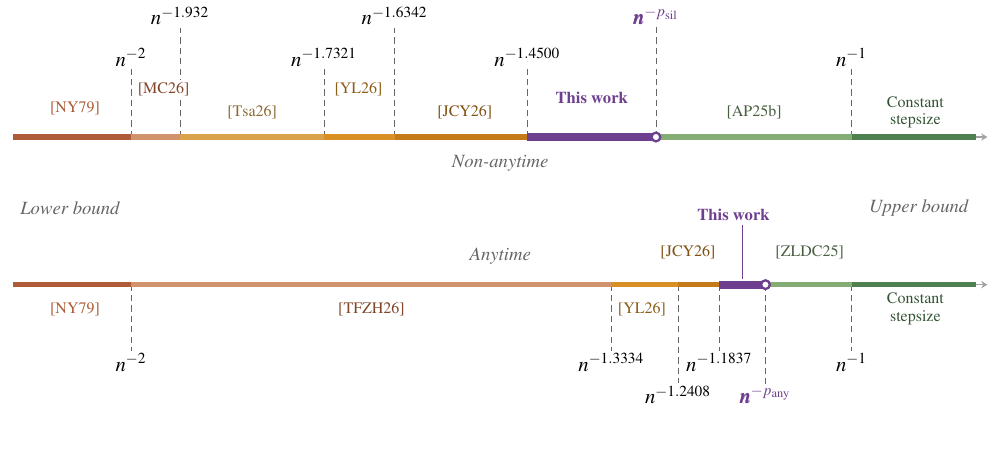}
  \caption{Lower and upper bounds for GD with predetermined stepsizes.}
  \label{fig:progress}
\end{figure}

\subsection{\texorpdfstring{Additional Literature Review}{Additional Literature Review}}
\label{sec:related-work}

\paragraph{Early work on quadratic objectives.}
For a quadratic objective $f(x)=\tfrac12(x-x_*)^\top A(x-x_*)$, the error $x_n-x_*$ is given by a polynomial in $A$ applied to $x_0-x_*$.
When the spectrum of $A$ lies in $[\mu,L]$ with $\mu>0$, the choice of predetermined stepsizes reduces to a classical polynomial minimax problem~\citep{Gavurin1950,GolubVarga1961}.
Young~\citep{Young1953} used Chebyshev stepsizes to attain the optimal $O(\sqrt\kappa\log(1/\varepsilon))$ iteration bound for reducing the squared distance to the minimizer by a factor of $\varepsilon$, where $\kappa=L/\mu$.
Conjugate gradients~\citep{HestenesStiefel1952} and the heavy-ball method~\citep{Polyak1964} achieve acceleration on quadratic objectives through recurrences that also use previous iterates.
This polynomial analysis relies on the constant Hessian of quadratic objectives and does not directly extend to general smooth strongly convex functions~\citep[Chapter~8]{Altschuler2018}.

\paragraph{Accelerating GD via stepsize schedules.}
Beyond quadratics, Altschuler's thesis~\citep[Chapter~8]{Altschuler2018} established improvements over short horizons in the strongly convex setting.
Later, Grimmer~\citep{Grimmer2024} proved that periodic long steps improve the function suboptimality bound by a constant factor when the initial level set is bounded.
The silver schedules discussed above improve the polynomial exponent in both the strongly convex and general smooth convex settings~\citep{AltschulerParrilo2025HedgingI,AltschulerParrilo2025}.
In concurrent work on smooth convex optimization, Grimmer, Shu, and Wang~\citep{GrimmerShuWang2023LongSteps} proved an $O(n^{-1.0564})$ bound, slightly improving on the classical $O(n^{-1})$ rate, under additional assumptions.\footnote{This result applies to the smallest function suboptimality among the iterates, $\min_{0\le t\le n}\{f(x_t)-f(x_*)\}$.
It also assumes that the set $\{x:f(x)\le f(x_0)\}$ is bounded.}
For the original smooth convex silver schedule,~\citep{WangMaYangZhou2024} determine the exact worst-case bound on function suboptimality at horizons $n=2^k-1$.
As follow-up work on the silver schedule,~\citep{GrimmerShuWang2025Composing,ZhangJiang2026} developed composition and concatenation rules that improve the constant factor and yield the silver rate at arbitrary prescribed horizons.
Under the symmetric composition rules in these works,~\citep{LiuChenJiangWang2026} proved that recursively combining schedules whose lengths differ by at most one minimizes the worst-case bound on function suboptimality at every horizon.
\citep{GrimmerShuWang2025Composing} also proved that the silver rate is minimax optimal within a family of schedules generated by their recursive composition rules.
The silver-rate guarantee has also been extended to proximal and projected GD~\citep{BokAltschuler2026}.
For separable smooth strongly convex objectives, Altschuler and Parrilo~\citep{AltschulerParrilo2024RandomStepsizes} obtained asymptotic acceleration using random stepsizes.

Acceleration through stepsize choice has also been studied in learning problems and min-max optimization.
For linear regression, periodic schedules combining short and long steps accelerate GD when the Hessian eigenvalues form two clusters~\citep{Oymak2021}.
For linearly separable logistic regression, a large constant stepsize chosen based on the horizon length gives an $\widetilde O(n^{-2})$ loss bound~\citep{WuBartlettTelgarskyYu2024}.
Large constant stepsizes also improve convergence guarantees for certain two-layer networks trained with logistic loss~\citep{CaiWuMeiLindseyBartlett2024} and for regularized logistic regression on separable data under suitable regularization conditions~\citep{WuMarionBartlett2025}.
More recently, GD with a predetermined stepsize schedule achieves stretched exponential loss decay for separable logistic regression~\citep{KaleManupriyaMarionBachRaj2026}.
Beyond minimization, power-law schedules accelerate the last-iterate convergence of the extragradient method on unconstrained biaffine min-max problems~\citep{WuZhengCaiLuo2026}.

\paragraph{Performance estimation.}
As mentioned above, PEPs provide a systematic framework for analyzing and designing first-order optimization methods.
Optimizing a PEP guarantee over the method's coefficients can identify first-order methods with optimal convergence guarantees within a chosen class~\citep{DroriTaylor2020}.
Within the class of fixed-step methods,~\citep{GrimmerJoPark2026} proved that OGM uniquely minimizes worst-case function suboptimality.
Branch-and-bound methods have also been used to numerically solve the minimax problem for GD with horizons up to $n=25$, yielding predetermined stepsizes that minimize the worst-case function suboptimality over smooth convex functions~\citep[Section~6.1]{DasGuptaVanParysRyu2024}.
These numerical solutions closely agree with the recursively constructed schedules in~\citep{GrimmerShuWang2025Composing}, providing numerical support for the conjectured optimality of $\psil$.
For a broader perspective on minimax optimality, see the survey~\citep{GrimmerShuWang2026Survey}.

\section{Overview of Proof Idea}
\label{sec:overview}
In this section, we first review classical lower bound constructions and the recent hard functions of~\citep{MaChen2026,JungChoYun2026}, highlighting the ideas that guide our construction.
We then present the idea behind our hard function construction.\footnote{Replacing $f$ by $f/L$ and each $h_t$ by $Lh_t$ preserves the iterates and normalized function suboptimality, so we use $L=1$ in this section and the proofs.}

\subsection{\texorpdfstring{Classical Lower Bound Constructions}{Classical Lower Bound Constructions}}
\label{sec:classical-constructions}

\paragraph{The role of Huber and quadratic functions.}
To build intuition, consider two basic one-dimensional hard functions first.
A one-sided version of the classical Huber function~\citep{Huber1964} is defined by
\[
  H_\delta(z):=
  \begin{cases}
    0,&z\le0,\\
    z^2/2,&0\le z\le\delta,\\
    \delta z-\delta^2/2,&z\ge\delta,
  \end{cases}
  \qquad \delta>0.
\]
Its gradient is constant on the affine region $z\ge\delta$, so short steps make slow progress toward the minimum.
Next, consider the quadratic objective $f(x)=x^2/2$, where the update is $x_t=(1-h_t)x_{t-1}$, so a large step ($h_t>2$) overshoots the minimizer and increases the distance to it.
These examples yield the following elementary lower bounds.

\begin{fact}[Huber and quadratic lower bounds]
\label{fact:huber-quadratic}
Let $n\ge1$ and $H=(h_1,\ldots,h_n)\in[0,\infty)^n$.
\begin{enumerate}[label=(\roman*),leftmargin=*]
\item (see~\citep[Fact~A.1]{YeLiu2026}) If every step is short ($h_t\le2$), a Huber example gives
\[
  \mathcal R_n(H)\ge\frac{1}{1+2\sum_{t=1}^n h_t}
  \ge\frac{1}{4n+1}.
\]

\item If every step is long ($h_t>2$), the quadratic $f(x)=x^2/2$ gives
\[
  \mathcal R_n(H)\ge\prod_{t=1}^n(1-h_t)^2>1.
\]
\end{enumerate}
\end{fact}

Every constant schedule falls into one of these two cases and therefore has a worst-case error of $\Omega(n^{-1})$.
More precisely, the same examples give~\citep[Section~3.2]{DroriTeboulle2014}
\[
  \mathcal R_n(h,\ldots,h)\ge
  \max\left\{\frac{1}{1+2nh},(1-h)^{2n}\right\},\qquad h\ge0.
\]
For all $0<h<2$, this lower bound matches the upper bound exactly, including the constant factor~\citep{Kim2024ExactGD}.

These examples suggest that a hard function for GD should combine affine regions that delay progress during short steps with quadratic curvature that causes large steps to overshoot the minimizer.\footnote{\citep[proof of Theorem~2]{KornowskiShamir2024} made an early attempt in this direction.}
This insight also appears in PEP analyses~\citep{DroriTeboulle2014} and motivates hedging between short and long steps in the silver schedule~\citep[Section~2.1]{AltschulerParrilo2025HedgingI}.

\paragraph{The role of zero-chains.}
As discussed above, using the Huber function alone yields a lower bound that weakens as the sum of stepsizes grows.
A different approach uses higher-dimensional hard functions to ensure the minimizer has non-zero components in unexplored directions.
Nemirovsky and Yudin~\citep{NemirovskyYudin1979} used this idea to establish the classical $\Omega(n^{-2})$ lower bound for general deterministic first-order methods.
Under the linear-span assumption\footref{fn:first-order-models}, Nesterov~\citep[Section~2.1.2]{Nesterov2004} presents a simpler quadratic zero-chain that limits how quickly a method can explore new coordinates, regardless of its stepsizes.
This chain couples adjacent coordinates $x^{(j)}$ and $x^{(j+1)}$ through quadratic terms, where $x^{(j)}$ denotes the $j$th coordinate of $x$.

\begin{fact}[{Nesterov's quadratic chain~\protect\citep[Section~2.1.2]{Nesterov2004}}]
\label{fact:classical-zero-chain}
Fix $n\ge1$ and let $d=2n+1$.
The function
\[
 f(x)=\frac18\left[(1-x^{(1)})^2+
 \sum_{j=1}^{d-1}(x^{(j)}-x^{(j+1)})^2+(x^{(d)})^2\right]
\]
is convex and $1$-smooth.
Starting from $x_0=0$, every method satisfying the linear-span assumption obeys
\[
 \frac{2[f(x_n)-f(x_*)]}{\|x_0-x_*\|^2}
 \ge\frac{3}{16(n+1)^2}.
\]
In particular, this holds for GD with any choice of stepsizes.
\end{fact}

The bound follows from the \emph{first-order zero-chain} property of $f$~\citep[Definition~1]{DroriTaylor2022}.
This means that, for every $0\le r<d$,
\[
  x^{(r+1)}=\cdots=x^{(d)}=0
  \quad\Longrightarrow\quad
  (\nabla f(x))^{(j)}=0\quad\text{for }j\ge r+2.
\]
Starting from $x_0=0$, each iteration of a method satisfying the linear-span assumption can therefore introduce at most one new coordinate, so $x_t^{(j)}=0$ for $j>t$.
The minimizer, however, has $x_*^{(j)}=1-j/(d+1)>0$ in every coordinate.
Minimizing $f$ with $x^{(n+1)}=\cdots=x^{(d)}=0$ and comparing the result with $f(x_*)$ gives the lower bound in \cref{fact:classical-zero-chain}~\citep[Theorem~2.1.7]{Nesterov2004}.

\subsection{\texorpdfstring{Stepsize-Dependent Lower Bounds}{Stepsize-Dependent Lower Bounds}}
\label{sec:recent-constructions}

The classical $\Omega(n^{-2})$ lower bound discussed above holds for general deterministic first-order methods~\citep{NemirovskyYudin1979}.
To obtain tighter lower bounds for GD, we must therefore tailor the hard function to the prescribed stepsize schedule.
Based on this idea, \citep{MaChen2026} introduced a construction that changes gradient direction after selected steps.
These steps are chosen adversarially for each prescribed schedule.
Their construction eventually yields an $\Omega(n^{-1.932})$ lower bound, which is the first improvement over the classical $\Omega(n^{-2})$ bound in this setting.

\paragraph{Selecting the checkpoints.}
To formalize this selection, fix a nonnegative schedule $H=(h_1,\ldots,h_n)$ and select a set of update indices $T=\{t_1<\cdots<t_k\}$ called \emph{checkpoints}. The updates between successive checkpoints form a \emph{gap}.
Set $t_0=0$, $t_{k+1}=n+1$, and define
\begin{equation}
  b_i:=h_{t_i}\quad(1\le i\le k),
  \qquad
  s_i:=\sum_{t=t_{i-1}+1}^{t_i-1}h_t\quad(1\le i\le k+1).
  \label{local:checkpoint-definitions}
\end{equation}
Here $b_i$ is the stepsize at checkpoint $t_i$, and $s_i$ is the \emph{mass} of the preceding gap, the sum of its stepsizes. The final gap has mass $s_{k+1}$.
\Cref{fig:checkpoint-gaps} illustrates this notation.

\begin{figure}[!htbp]
  \centering
  \includegraphics[width=\textwidth]{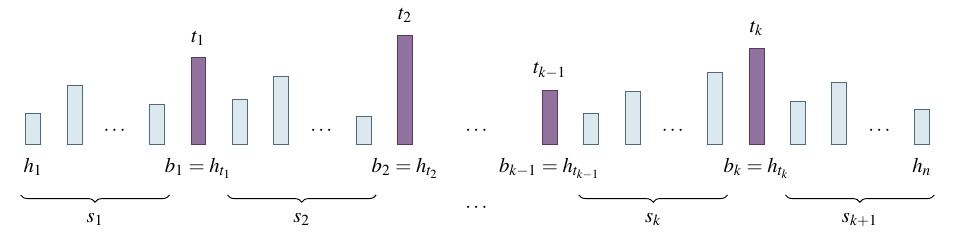}
  \caption{The selected stepsizes are $b_i=h_{t_i}$. Each brace marks the sum $s_i$ of the stepsizes in a gap.}
  \label{fig:checkpoint-gaps}
\end{figure}

For selected long steps, the hard function in~\citep[Theorem~4.1]{MaChen2026} keeps the gradient constant during each gap and changes it after each selected step.
See also~\citep[Appendix~A]{YeLiu2026} for a geometric illustration of the trajectory.
However, their hard function is defined globally on all coordinates, which limits the flexibility of the construction.
To make the construction more flexible and possibly yield a stronger bound, \citep{JungChoYun2026} instead use a sum of \emph{local hard functions}, which we also call \emph{components}.
Each component couples only two adjacent coordinates.
As in the construction of~\citep{MaChen2026}, the gradient changes direction after selected steps.
During each gap, the current component contributes a constant gradient, while all later components remain inactive.\footnote{A component is \emph{active} at an iterate if its gradient there is nonzero, and \emph{inactive} otherwise.}
This chain resembles Nesterov's quadratic construction and leads to the stronger $\Omega(n^{-\log_2(1+\sqrt3)})$ lower bound, where $\log_2(1+\sqrt3)\approx1.4500$.
For a fixed checkpoint set, this gives the following bound.

\begin{fact}[{\protect\citep[Lemma~2.3]{JungChoYun2026}}]
\label{fact:jcy-transfer}
For every positive schedule $H=(h_1,\ldots,h_n)$ and every checkpoint set $T$,
\begin{equation}
  \mathcal R_n(H)\ge
  \frac{1}{4(1+s_{k+1})}
  \prod_{i=1}^k
  \left[\frac{b_i}{2(2+s_i)}\right]^2.
  \label{eq:jcy-transfer}
\end{equation}
For $T=\varnothing$, the product is one and $s_1=\sum_{t=1}^n h_t$.
\end{fact}

Specifically, their construction uses the one-sided Huber function $H_\delta$ defined above, joining adjacent coordinates through
\begin{equation*}
  F(x):=\frac14\sum_{i=1}^k
  H_{\delta_i}\bigl(x^{(i)}-x^{(i+1)}-\ell_i\bigr)
  +\frac12H_{\delta_{k+1}}\bigl(x^{(k+1)}-\ell_{k+1}\bigr).
\end{equation*}
Here $\ell_i$ is the \emph{activation threshold} of component $i$, and $D_i:=x_{t_{i-1}}^{(i)}-\ell_i>0$ is the \emph{excess} above it at the start of gap $i$.
Starting from $x_0=e_1$, set $\ell_1=0$ and $D_1=1$.
For this Huber construction, the parameters are chosen recursively as $\delta_i=\frac{2D_i}{2+s_i}$, $\ell_{i+1}=\frac{s_i\delta_i}{4}$, and $D_{i+1}=\frac{b_i\delta_i}{4}$ for $1\le i\le k$.
During gap $i$, coordinate $i+1$ never exceeds $\ell_{i+1}$ and equals it just before the checkpoint.
The checkpoint then raises this coordinate to $\ell_{i+1}+D_{i+1}$ and activates component $i+1$, while all later components remain inactive.
Thus each factor $b_i/[2(2+s_i)]$ in \eqref{eq:jcy-transfer} is the \emph{transfer ratio} $D_{i+1}/D_i$.
Since $D_1=1$, their product is the final excess $D_{k+1}$.
The final one-dimensional component is $\frac12H_{\delta_{k+1}}(x^{(k+1)}-\ell_{k+1})$.
Choosing $\delta_{k+1}=D_{k+1}/(1+s_{k+1})$ gives a lower bound proportional to $D_{k+1}^2/(1+s_{k+1})$.
See~\citep[Section~2]{JungChoYun2026} for details.

For a prescribed schedule $H$, this Huber construction is available for every checkpoint set $T$.
The checkpoints can therefore be chosen to maximize the right-hand side of \eqref{eq:jcy-transfer}.
The recursive argument in~\citep[Section~3]{JungChoYun2026} bounds this maximum from below by $\Omega(n^{-\log_2(1+\sqrt3)})$, uniformly over schedules $H$.

Another key insight from \cref{fact:jcy-transfer} is that the bound favors choosing long steps as checkpoints, placing their stepsizes $b_i$ in the numerator and the gap sums $s_i$ in the denominator.
This matches the lower bound intuition in \cref{sec:classical-constructions} that short steps should move the iterate slowly in one direction, while long steps should change its direction or cause it to overshoot.

\paragraph{The bottleneck.}
For fixed $H$ and $T$, improving the local functions to give larger transfer ratios $D_{i+1}/D_i$ would further strengthen the lower bound.
In general, the argument in~\citep[Section~3 and Appendix~C]{JungChoYun2026} suggests a lower bound of order $n^{-\log_2(1+\sqrt{1+c})}$ when the coefficient of $s_i$ in the denominator of the transfer ratio $D_{i+1}/D_i$ is $c$.
The transfer ratio $b_i/[2(2+s_i)]$ in \eqref{eq:jcy-transfer} has $c=2$, whereas reducing $c$ toward $1$ would bring the exponent toward $\psil$.
During gap $i$, the current Huber component contributes equally to the decrease in $x^{(i)}$ and the increase in $x^{(i+1)}$.
Both changes reduce the margin $x^{(i)}-x^{(i+1)}-\ell_i$, producing the coefficient $2$ of $s_i$ in the denominator $4+2s_i$ in \eqref{eq:jcy-transfer}~\citep[Section~2]{JungChoYun2026}.

\subsection{\texorpdfstring{Our Improvement: Bending the Local Trajectory}{Our Improvement: Bending the Local Trajectory}}
\label{sec:bending-overview}
Keeping the chain structure and notation above, we seek local hard functions $\Phi_i$ for an objective of the form
\[
  F(x):=\frac12\sum_{i=1}^k\Phi_i(x)
       +\frac12\Phi_{k+1}(x).
\]
For $1\le i\le k$, each $\Phi_i$ should depend only on $x^{(i)}$ and $x^{(i+1)}$ and be nonnegative, convex, and $1$-smooth.
The final component $\Phi_{k+1}$ depends only on $x^{(k+1)}$.
For $1\le i\le k$, we impose the following \emph{transfer property} so that only the current component $\Phi_i$ contributes to the updates of $x^{(i)}$ and $x^{(i+1)}$ during gap $i$ and in its checkpoint update\footnote{In~\citep[Section~2 and Appendix~B.2]{JungChoYun2026}, the preceding component $\Phi_{i-1}$ contributes nonnegatively to the update of $x^{(i)}$ for $i>1$, helping keep $\Phi_i$ in its affine region with constant gradient during gap $i$.
Such a contribution can change the gradient direction in our construction, so we impose $\partial_{x^{(i+1)}}\Phi_i(x)=0$ whenever $x^{(i+1)}\ge\ell_{i+1}$.}
\begin{equation}
  \begin{aligned}
    \Phi_i(x)=0,\quad \nabla\Phi_i(x)=0
    &\qquad\text{whenever }x^{(i)}\le\ell_i\text{ and }x^{(i+1)}\ge0,\\
    \partial_{x^{(i+1)}}\Phi_i(x)=0
    &\qquad\text{whenever }x^{(i+1)}\ge\ell_{i+1}.
  \end{aligned}
  \label{eq:transfer-property}
\end{equation}
The first condition keeps later components inactive while their first coordinates stay at or below their thresholds and their second coordinates stay nonnegative.
For $i>1$, applying the second condition to $\Phi_{i-1}$ gives $\partial_{x^{(i)}}\Phi_{i-1}=0$ whenever $x^{(i)}\ge\ell_i$, so the preceding component does not affect the update of $x^{(i)}$.
The shaded region in \cref{fig:gradient-turning} illustrates the second condition for $\Phi_i$.
For $i<k$, the iterates remain in this region during gap $i+1$, including just before the checkpoint update at $t_{i+1}$.

As before, we write $D_{i+1}:=x_{t_i}^{(i+1)}-\ell_{i+1}>0$ for the excess above $\ell_{i+1}$ immediately after the checkpoint update at $t_i$.
Our goal is to construct $\Phi_i$ with transfer ratio
\[
  \frac{D_{i+1}}{D_i}\approx \frac{b_i}{cs_i}.
\]
We aim to make $c$ arbitrarily close to $1$.

Notice that growth in $x^{(i+1)}$ during gap $i$ raises the threshold $\ell_{i+1}$ needed to keep $\Phi_{i+1}$ inactive.
To leave a larger excess $D_{i+1}$ after checkpoint $t_i$, we limit this early growth while ensuring a large increase at the checkpoint update.
When $s_i$ is sufficiently large, the trajectory in gap $i$ initially moves mainly in coordinate $x^{(i)}$ and turns toward coordinate $x^{(i+1)}$ before checkpoint $t_i$.
We call this change in direction \emph{bending}, the key new ingredient in our construction.

More concretely, we construct the GD trajectory in gap $i$ with gradient vectors on a circular arc.
For parameters $\epsilon,\gamma>0$, consider the circle centered at $(0,\gamma)$ with radius $R:=\sqrt{\epsilon^2+(1+\gamma)^2}$, which passes through $(\epsilon,-1)$.
Its intersection with the positive horizontal axis is $(c,0)$, where $c=\sqrt{1+2\gamma+\epsilon^2}$.
Up to a common positive scale, we place the gradients on the arc joining these two points, as shown on the left of \cref{fig:gradient-turning}.
They turn toward $(\epsilon,-1)$ during gap $i$.
At $x_{t_i-1}$, just before the checkpoint update, the gradient is proportional to $(\epsilon,-1)$, so the update mainly increases $x^{(i+1)}$ when $\epsilon$ is small.

Given the stepsizes in gap $i$, we construct the trajectory backward from the iterate just before checkpoint $t_i$; see \cref{lem:local-backward}.
We use the arc vectors along this trajectory, together with $0$ and $(c,0)$, as slopes in a finite maximum of affine functions.
After scaling and translation, the Moreau envelope of this function gives a convex, $1$-smooth local hard function $\Phi_i$ whose gradient at each constructed iterate is the corresponding scaled arc vector; see \cref{lem:local-transfer} and \eqref{local:chain-component}.\footnote{A similar use of Moreau envelopes to realize a prescribed GD trajectory appears in~\citep[Section~4.3]{MaChen2026}.}

\begin{figure}[!htbp]
  \centering
  \includegraphics[width=\textwidth]{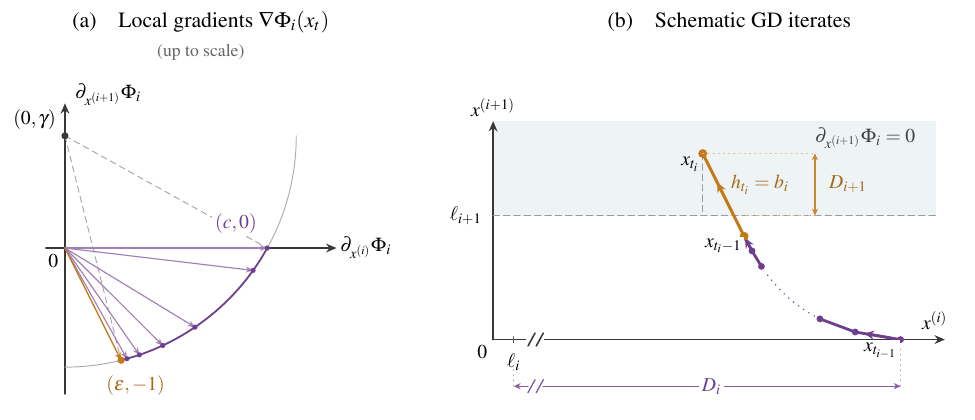}
  \caption{Local gradients $\nabla\Phi_i(x_t)$ (left) and a schematic GD trajectory (right). The step $h_{t_i}$ (orange) activates the next component.}
  \label{fig:gradient-turning}
\end{figure}

\paragraph{The improved hard-function bound.}
The local hard function constructed in \cref{lem:local-transfer} gives the transfer ratio $D_{i+1}/D_i=(b_i-2)/[c(a+s_i)]$, where $a$ depends only on $\epsilon,\gamma$.
To leave positive excess, we use checkpoints with $b_i>2$, which we call \emph{admissible}.
A checkpoint set is admissible if all its checkpoints are admissible.
Applying these local constructions to each gap and choosing the final component $\Phi_{k+1}$ gives the following lower bound.

\begin{samepage}
\begin{lemma}[Lower bound for a checkpoint set]
\label{thm:transfer}
Fix $\epsilon,\gamma>0$ and let $c:=\sqrt{1+2\gamma+\epsilon^2}$ and $R:=\sqrt{\epsilon^2+(1+\gamma)^2}$.
For every nonnegative schedule $H=(h_1,\ldots,h_n)$ and admissible checkpoint set $T$,
\begin{equation}
  \mathcal R_n(H)\ge
  \frac{1}{4(1+s_{k+1})}
  \prod_{i=1}^k
  \left[\frac{b_i-2}{c(a+s_i)}\right]^2,
  \label{eq:transfer}
\end{equation}
where $a=\max\bigl\{2,\frac{2}{c}\bigl(\epsilon+\epsilon^{-1}[\frac{c(1+\gamma+R)}{\epsilon(\gamma+R)}]^{R/\gamma}\bigr)\bigr\}$.
\end{lemma}
\end{samepage}

\paragraph{From the construction to the main lower bounds.}
It remains to choose checkpoints that make \eqref{eq:transfer} give a uniform lower bound over all schedules.
Following~\citep[Section~3 and Appendix~C]{JungChoYun2026}, we derive a checkpoint cost from \eqref{eq:transfer} and minimize it over admissible checkpoint sets for each schedule.
The structure of a schedule that maximizes this minimum allows us to split at a checkpoint and apply induction to the two shorter sequences.
The resulting recursion gives an $\Omega(n^{-p})$ lower bound for every $p>\psil$ by taking $c$ sufficiently close to $1$.
The explicit constant $a$ in \cref{thm:transfer} tends to infinity as $c\downarrow1$.
Balancing this growth against the improvement in $p$ gives \cref{thm:nonanytime}, as shown in \cref{sec:global}.

In the anytime setting, \citep{TsaiFatkhullinZhangHe2026} first improved on the classical $\Omega(n^{-2})$ lower bound, before the corresponding non-anytime breakthrough in~\citep{MaChen2026}.
They proved an $\Omega(n^{-4/3})$ bound and developed a proof strategy that we adapt below.
For our anytime result, bounded stepsizes already give an $\Omega(n^{-1})$ Huber lower bound.
Otherwise, following~\citep[Section~4.1]{TsaiFatkhullinZhangHe2026}, we examine horizons at which the last stepsize exceeds all earlier ones.
The Huber bound is strong when the total stepsize is small, while selecting a large last step as a checkpoint gives a complementary bound from \eqref{eq:transfer}.
Together with the estimate on the total stepsize at these horizons in \cref{any:lem:record-sum}, these two bounds yield $\Omega(n^{-2p/(1+p)})$ infinitely often.
\citep[Section~4]{JungChoYun2026} also used this framework to obtain an anytime lower bound from their non-anytime hard function construction, with the same relation between the two exponents.
Choosing $p$ to approach $\psil$ as the horizon grows yields \cref{thm:anytime}; see \cref{app:anytime}.

\section[Proof of the Main Theorems]{Proof of the Main Theorems}
\label{sec:nonanytime}

In this section, we first construct the full hard function $F$ in \eqref{local:global-objective} to prove \cref{thm:transfer}, and then prove the main lower bound theorems (\cref{thm:nonanytime,thm:anytime}).
The full proofs are in the appendices.

\subsection{Proof of Lemma~\ref{thm:transfer}}
\label{sec:local}

\paragraph{Constructing the gap trajectory.}
Using the parameters $\epsilon,\gamma,c,R$ from \cref{thm:transfer}, we first construct a two-dimensional gap trajectory.
After scaling and translation, this will give the GD trajectory generated by the local hard function $\Phi_i$ during gap $i$.
We choose the gradient vectors on a circular arc with center $(0,\gamma)$ and radius $R$.
Let $K$ be the convex hull of the origin and the shorter arc joining $(c,0)$ to $(\epsilon,-1)$.
\Cref{lem:local-backward} constructs points $r_j$ backward from $(\epsilon,-1)$ so that $r_j=r_{j-1}-\alpha_jv_{j-1}$, where $v_j=\Pi_K(r_j)$ is the Euclidean projection of $r_j$ onto $K$.
Each update has the form of a GD step with stepsize $\alpha_j$ and gradient $v_{j-1}$.
The projection condition later lets us realize the scaled vectors as gradients of a convex, $1$-smooth function $\Phi_i$.

\begin{lemma}[Gap trajectory]
\label{lem:local-backward}
For every $m\ge0$ and nonnegative numbers $\alpha_1,\ldots,\alpha_m$, there exist points $r_0,\ldots,r_m$ and arc vectors $v_0,\ldots,v_m$ satisfying
\[
  r_m=v_m=(\epsilon,-1),\qquad
  r_j=r_{j-1}-\alpha_jv_{j-1}\quad(1\le j\le m),
  \qquad v_j=\Pi_K(r_j)\quad(0\le j\le m).
\]
With $S:=\sum_{j=1}^m\alpha_j$, the initial point has the form
\[
  r_0=(\epsilon+P,-1-Q),\qquad
  0\le P\le cS,\qquad 0\le Q\le\left[\frac{c(1+\gamma+R)}{\epsilon(\gamma+R)}\right]^{R/\gamma}-1.
\]
\end{lemma}

Here $P$ is the decrease in the first coordinate and $Q$ is the increase in the second from $r_0$ to $r_m$.
When defining $\Phi_i$, we scale and translate these two coordinates to $x^{(i)}$ and $x^{(i+1)}$.
The same $c$ in $P\le cS$ will appear in the denominator $c(a+s_i)$ of the transfer ratio $D_{i+1}/D_i$.
The proof and a geometric illustration are given in \cref{local:app-backward}; see \cref{fig:backward-similar-triangles}.

\paragraph{Constructing the local hard function.}
For gap stepsizes $h_1,\ldots,h_m$, we apply \cref{lem:local-backward} with $\alpha_j=h_j/2$, then scale and translate its points so that the trajectory starts at $(1,0)$.
We then construct a nonnegative convex, $1$-smooth function $f$ such that GD on $f/2$ follows the scaled points.
The following lemma ensures the transfer property \eqref{eq:transfer-property} and gives the excess after the checkpoint update.

\begin{lemma}[Local hard function]
\label{lem:local-transfer}
Let $m\ge0$, $h_1,\ldots,h_m\ge0$, $s=\sum_{j=1}^m h_j$, and $b>2$.
There exist a nonnegative convex $1$-smooth function $f:\R^2\to\R$ and a threshold $\tau>0$ satisfying

\begin{equation}
  f(x,y)=0\quad(x\le0,\ y\ge0),\qquad
  \partial_2f(x,y)=0\quad(x\in\R,\ y\ge\tau).
  \label{local:zero-regions}
\end{equation}
Starting from $(x_0,y_0)=(1,0)$, GD on $f/2$ with stepsizes $h_1,\ldots,h_m,b$ satisfies
\[
\begin{gathered}
  x_j>0,\qquad 0\le y_j\le y_m=\tau-\frac{2}{c(a+s)}\quad(0\le j\le m),\\
  y_{m+1}=\tau+\frac{b-2}{c(a+s)},
\end{gathered}
\]
where $a$ is defined in \cref{thm:transfer}.
\end{lemma}
During the gap, the second coordinate stays below $\tau$.
The checkpoint update leaves excess $\frac{b-2}{c(a+s)}$, which becomes the transfer ratio $D_{i+1}/D_i$ in the chain.
The proof is given in \cref{local:app-transfer}.

\begin{proof}[Proof of \cref{thm:transfer} (Constructing the global hard function)]
We now combine the local hard functions along the coordinate chain.
For gap $i$, apply \cref{lem:local-transfer} to $h_{t_{i-1}+1},\ldots,h_{t_i-1}$ and the checkpoint stepsize $b_i$.
Write $f_i,\tau_i$ for the resulting function and threshold. Set
\begin{equation}
  D_1:=1,\quad\ell_1:=0,\qquad
  D_{i+1}:=D_i\frac{b_i-2}{c(a+s_i)},\quad\ell_{i+1}:=D_i\tau_i\quad(1\le i\le k).
  \label{local:chain-parameters}
\end{equation}
Define the local hard function
\begin{equation}
  \Phi_i(x):=D_i^2 f_i\left(
    \frac{x^{(i)}-\ell_i}{D_i},\frac{x^{(i+1)}}{D_i}\right).
  \label{local:chain-component}
\end{equation}
Thus the local initial point $(1,0)$ corresponds to $(\ell_i+D_i,0)$ in the chain.
Under the scaling in \eqref{local:chain-component}, a GD step on $\Phi_i/2$ becomes a step of the same size on $f_i/2$ in the local coordinates $\bigl((x^{(i)}-\ell_i)/D_i,\,x^{(i+1)}/D_i\bigr)$.
The two conditions in \eqref{local:zero-regions} give the transfer property \eqref{eq:transfer-property}, since a nonnegative differentiable function has zero gradient wherever its value is zero.
Set $\delta=D_{k+1}/(1+s_{k+1})$ and define the final component by $\Phi_{k+1}(x):=H_\delta(x^{(k+1)}-\ell_{k+1})$.
The full objective is
\begin{equation}
  F(x):=\frac12\sum_{i=1}^k\Phi_i(x)
       +\frac12\Phi_{k+1}(x).
  \label{local:global-objective}
\end{equation}
The final Huber term also has zero gradient when $x^{(k+1)}\le\ell_{k+1}$.
Each component is $1$-smooth and each coordinate appears in at most two terms, so $F$ is $1$-smooth (see \cref{local:app-chain}).

For gap $i$, set $m=t_i-t_{i-1}-1$ and let $(x_j,y_j)$ be the GD iterates on $f_i/2$ given by \cref{lem:local-transfer}.
Starting from the initial point $e_1$, we verify by induction over the gaps and their updates that, for $0\le j\le m+1$, all coordinates of $x_{t_{i-1}+j}$ beyond $i+1$ are zero and
\begin{equation}
  \bigl(x_{t_{i-1}+j}^{(i)},x_{t_{i-1}+j}^{(i+1)}\bigr)
  =\bigl(\ell_i+D_i x_j,D_i y_j\bigr).
  \label{local:coordinate-map}
\end{equation}
In each gap, we check the coordinate bounds needed to apply the transfer property \eqref{eq:transfer-property} to the preceding and later components.
Suppose gap $i$ starts with $x^{(i)}=\ell_i+D_i$ and all later coordinates zero, as holds for $i=1$.
By \cref{lem:local-transfer}, $\ell_i+D_i x_j>\ell_i$ and $0\le D_i y_j<\ell_{i+1}$ for $0\le j\le m$.

Before an update with $0\le j\le m$, suppose \eqref{local:coordinate-map} holds and all coordinates beyond $i+1$ are zero.
The second condition in \eqref{eq:transfer-property}, applied to $\Phi_{i-1}$, then gives $\partial_{x^{(i)}}\Phi_{i-1}=0$ when $i>1$.
Since $x^{(i+1)}<\ell_{i+1}$ and all coordinates beyond $i+1$ are zero, the first condition in \eqref{eq:transfer-property} and the analogous property of the final Huber term keep the next and all later components inactive.
Thus only $\Phi_i$ contributes to the updates of $x^{(i)}$ and $x^{(i+1)}$, and no component changes the coordinates beyond $i+1$.
By the scaling in \eqref{local:chain-component}, the next GD step gives \eqref{local:coordinate-map} at $j+1$ and leaves all coordinates beyond $i+1$ zero.
Induction over the updates, including the checkpoint update, gives \eqref{local:coordinate-map} at $j=m+1$. By \cref{lem:local-transfer} and \eqref{local:chain-parameters},
\[
  x_{t_i}^{(i+1)}=D_i y_{m+1}
  =D_i\tau_i+D_i\frac{b_i-2}{c(a+s_i)}
  =\ell_{i+1}+D_{i+1},
\]
with all coordinates beyond $i+1$ still zero.
This is the starting state for gap $i+1$.

At $x_{t_k}$, the last coordinate has excess $D_{k+1}$ above $\ell_{k+1}$.
During the remaining steps it stays above this threshold, so the preceding component contributes zero to its update when $k>0$.
Since all other components are nonnegative, the final Huber calculation in \cref{local:app-chain} gives
\[
  F(x_n)-\min F
  \ge\frac{D_{k+1}^2}{4(1+s_{k+1})}
  =\frac{1}{4(1+s_{k+1})}
    \prod_{i=1}^k
    \left[\frac{b_i-2}{c(a+s_i)}\right]^2.
\]
Since $x_0=e_1$ and the origin minimizes $F$, the initial distance is $1$.
The normalization in \eqref{eq:worst-case-error} gives $\mathcal R_n(H)\ge2F(x_n)$, which implies \eqref{eq:transfer}.
\end{proof}

\subsection[(Almost) Tight Lower Bounds]{(Almost) Tight Lower Bounds}
\label{sec:global}
\label{sec:silver}

With these preparations in hand, we now prove \cref{thm:nonanytime}.
We first give a high-level overview of the argument.
For a fixed stepsize schedule $H$, each admissible checkpoint set gives a hard function through \cref{thm:transfer}, and we choose checkpoints to make the resulting lower bound as large as possible.
We seek a lower bound on this maximum that holds for every schedule $H$.

Fix $0<\epsilon<1/2$ and $\gamma>0$, set $c:=\sqrt{1+2\gamma+\epsilon^2}$, and take
$a$ as in \cref{thm:transfer}.
For a schedule $H=(h_1,\ldots,h_n)$ and an admissible checkpoint set $T$ (so $b_i>2$ at every checkpoint),
choose a parameter $\Lambda$ and, following~\citep[Section~3 and Appendix~C]{JungChoYun2026}, define the \emph{checkpoint cost} by
\begin{equation}
 Q_\Lambda(T;H):=c(\Lambda+s_{k+1})
 \prod_{i=1}^k\frac{c(a+s_i)}{b_i-2},
 \qquad \Lambda\ge a.
 \label{global:eq-checkpoint-expression}
\end{equation}
Here $\Lambda$ replaces the constant $a$ in the factor for the gap after the last checkpoint.
The lower bound \eqref{eq:transfer} implies
\begin{equation}
 \mathcal R_n(H)\ge\frac{c^2(\Lambda-1)}{Q_\Lambda(T;H)^2}.
 \label{global:eq-cost-lower-bound}
\end{equation}
The proof is provided in \cref{global:app-cost-lower-bound}.
%We keep $\Lambda$ free so that the later choice $\Lambda=an^p$ supplies a factor $n^p$ in the numerator, yielding an $n^{-p}$ lower bound from an $O(n^p)$ upper bound on the cost.
For fixed $H$ and $\Lambda$, maximizing the right-hand side of \eqref{global:eq-cost-lower-bound} is equivalent to minimizing $Q_\Lambda(T;H)$ over checkpoint sets.
We therefore define
\[
 V_\Lambda(H):=\min_{T\text{ admissible}} Q_\Lambda(T;H).
\]
A checkpoint set attaining $V_\Lambda(H)$ is called \emph{tight}.
For fixed $H$, this minimum admits the following recursive decomposition.
\begin{lemma}
\label{global:lem-value-decomposition}
For every $H\in[0,\infty)^n$, $n\ge1$, and $\Lambda\ge a$,
\begin{equation}
 V_\Lambda(H)=\min\left\{
 c\left(\Lambda+\sum_{j=1}^n h_j\right),\;
 \min_{\substack{1\le t\le n\\t\text{ admissible}}}
 \frac{V_a(H_{1:t-1})V_\Lambda(H_{t+1:n})}{h_t-2}
 \right\}.
 \label{global:eq-value-decomposition}
\end{equation}
Here $H_{i:j}:=(h_i,\ldots,h_j)$, $V_\mu(\varnothing)=c\mu$, and the inner minimum is $+\infty$ if there is no admissible checkpoint.
\end{lemma}

\begin{proof}
For an admissible checkpoint set $T$ and $t\in T$, let $T_L$ and $T_R$ be the selected indices before and after $t$, reindexed within the two subsequences.
Then
\[
 Q_\Lambda(T;H)=\frac{Q_a(T_L;H_{1:t-1})Q_\Lambda(T_R;H_{t+1:n})}{h_t-2}.
\]
Minimizing over $T_L,T_R$ and $t$ accounts for every nonempty checkpoint set. Including the empty set gives \eqref{global:eq-value-decomposition}.
\end{proof}

To make the lower bound uniform over schedules of length $n$, define
\[
 U_n(\Lambda):=\sup_{H\in[0,\infty)^n}V_\Lambda(H).
\]
Since $V_\Lambda(H)\le U_n(\Lambda)$, we have
\[
 \mathcal R_n(H)\ge\frac{c^2(\Lambda-1)}{V_\Lambda(H)^2}
 \ge\frac{c^2(\Lambda-1)}{U_n(\Lambda)^2}.
\]
Thus it suffices to bound $U_n(\Lambda)$ from above.
We call a schedule attaining $U_n(\Lambda)$ a \emph{maximizing schedule}.

\Cref{global:lem-extremal-schedule} shows that, for every maximizing schedule, there exists a tight singleton checkpoint set.
This structure lets us turn the recursive formula for $V_\Lambda(H)$ in \cref{global:lem-value-decomposition} into recursive bounds on $U_n(\Lambda)$, which is the key step in the proof.
\begin{lemma}[Maximizing schedules]
\label{global:lem-extremal-schedule}
For every $n\ge1$ and $\Lambda\ge a$, the supremum defining $U_n(\Lambda)$ is finite and attained. At every maximizing schedule, the
empty set and a singleton are tight.
\end{lemma}

For a tight singleton $\{t_*\}$ at a maximizing schedule, write $cW=U_n(\Lambda)$ and let $cx,cy$ be the empty-set costs of the left and right subproblems.
With $B=h_{t_*}-2$, equality of the empty-set and singleton costs gives
\begin{equation}\label{eq:empty-singleton}
     W=x+y+B-(a-2),\qquad WB=cxy.
\end{equation}
To see heuristically how these identities give rise to the silver exponent, consider a balanced split $x=y$. Neglect the fixed term $a-2$ in \eqref{eq:empty-singleton}, eliminating $B$ gives
\[
 W^2-2xW-cx^2=0,\qquad W=(1+\sqrt{1+c})x.
\]
Combining two equal subproblems therefore multiplies the cost by
$1+\sqrt{1+c}$. This suggests growth of order $n^{p(c)}$, where
$p(c)=\log_2(1+\sqrt{1+c})$, tending to $\psil$ as $c\downarrow1$.\footnote{The one-sided Huber construction in \eqref{eq:jcy-transfer} has $c=2$, and its lower bound $\Omega(n^{-\log_2(1+\sqrt3)})$~\citep{JungChoYun2026} has exponent $p(2)$.}

A rigorous argument in \cref{global:app-uniform-growth} based on \eqref{eq:empty-singleton} gives the following bound on $U_n(\Lambda)$.
\begin{lemma}
\label{thm:weighted-recursion}
Let $1<c\le2$, $p\in[\psil+c-1,3/2]$, and $\nu=1/p$.
Then, for every $n\ge0$ and $\Lambda\ge a$,
\begin{equation}
 \left[\frac{U_n(\Lambda)}c\right]^\nu
 \le a^\nu n+\Lambda^\nu.
 \label{eq:cost-growth}
\end{equation}
\end{lemma}

Now we are ready to prove \cref{thm:nonanytime}.
\begin{proof}[Proof of \cref{thm:nonanytime}]
Fix $0<\eta\le1/10$ and choose
\begin{equation}
 \begin{gathered}
 p:=\psil+\eta,\qquad \gamma:=\frac\eta4,\qquad
 \epsilon:=\left(\frac\eta{16}\right)^p,\\
 c:=\sqrt{1+2\gamma+\epsilon^2}.
 \end{gathered}
 \label{global:eq-parameter-choice}
\end{equation}
With $a$ as defined in \cref{thm:transfer}, set $\nu=1/p$ and $\Lambda=an^p$.
\Cref{global:app-parameters} verifies the parameter conditions of \cref{thm:weighted-recursion} and proves \eqref{eq:parameter-growth} below.
By \cref{thm:weighted-recursion} and $\Lambda^\nu=a^\nu n$,
$V_\Lambda(H)\le U_n(\Lambda)\le c(2a^\nu n)^p=c2^pan^p$.
Since $\Lambda-1\ge an^p/2$, \eqref{global:eq-cost-lower-bound} gives
\begin{equation}
 \mathcal R_n(H)\ge\frac{1}{2^{2p+1}a}\,n^{-p}
 \qquad\text{for every }H\in[0,\infty)^n.
 \label{eq:fixed-p-rate}
\end{equation}
Under the choice \eqref{global:eq-parameter-choice},
\begin{equation}
 \log a=O\!\left(\eta^{-1}\log(1/\eta)\right),
 \qquad 0<\eta\le1/10.
 \label{eq:parameter-growth}
\end{equation}
Consequently,
\[
 \mathcal R_n(H)\ge n^{-\psil}
 \exp\!\left[-\eta\log n-O\!\left(\eta^{-1}\log(1/\eta)\right)\right].
\]
For $n$ large enough, take $\eta=\sqrt{\log\log n/\log n}\le1/10$.
Then $\eta\log n=\sqrt{\log n\,\log\log n}$ and $\eta^{-1}\log(1/\eta)\le\sqrt{\log n\,\log\log n}$, so the preceding bound is at least $n^{-\psil-C\sqrt{\log\log n/\log n}}$ for an absolute constant $C$.
This proves \cref{thm:nonanytime}.
\end{proof}

The lower bound in \cref{thm:transfer} and the bound on $U_n(\Lambda)$ in \eqref{eq:cost-growth} also yield the anytime exponent $\pany=2\psil/(1+\psil)$.
The proof of \cref{thm:anytime} adapts the analysis of~\citep[Section~4]{JungChoYun2026} and is given in \cref{app:anytime}.

\section{Concluding Remarks}
\label{sec:conclusion}

We proved lower bounds for GD with predetermined nonnegative stepsizes that match the silver exponents in both the finite-horizon and anytime settings. Bending the gradient during each gap improves the transfer between local hard functions and closes the polynomial gap to the upper bounds.

The remaining open questions are:

\begin{enumerate}
\item Can the lower bounds match the silver exponents for schedules that may include negative stepsizes?
To our knowledge, the best known lower bounds for such schedules are $\Omega(n^{-1.6342})$ in the non-anytime case and $\Omega(n^{-1.2408})$ in the anytime case, as established in our earlier paper~\citep{YeLiu2026}.
Extending the lower bounds proved here to such schedules remains open, although we believe that our approach can be adapted to handle negative stepsizes.

\item Can the $\sqrt{\log\log n/\log n}$ losses in the exponents of \eqref{eq:main-nonanytime} and \eqref{eq:main-anytime} be reduced or removed?\footnote{We have obtained preliminary improvements by modifying our existing construction to choose gradient vectors along a logarithmic spiral rather than the circular arc in \eqref{local:arc}.}
In the circular construction illustrated in \cref{fig:backward-similar-triangles}, setting $\gamma=0$ would make every gradient vector $v_j$ equal to $(\epsilon,-1)$.
A local calculation similar to the Huber analysis in \cref{sec:recent-constructions} gives the coefficient $\epsilon+\epsilon^{-1}\ge2$ for $s_i$.
This yields a suboptimal lower bound as in~\citep{JungChoYun2026}.
We therefore need $\gamma>0$, so $c=\sqrt{1+2\gamma+\epsilon^2}$ can approach $1$ but cannot equal $1$.
For this construction, balancing the improvement as $c\downarrow1$ against the growth of $a$ leaves the $o(1)$ loss in the exponent $\psil+o(1)$.
Removing this loss may require a more refined local construction, or even an entirely new approach to constructing the global hard function.
\end{enumerate}

\section*{Acknowledgement}
We thank Jason Altschuler and Pablo Parrilo for helpful feedback regarding the literature on stepsize-based acceleration for GD.
We are also grateful to Benjamin Grimmer for discussions during the \emph{Readings in Optimization} seminar at MIT and for reading and commenting on the manuscript.
\section*{AI Disclosure}

We carefully read the paper \citep{JungChoYun2026} and traced the difference between their exponent $\log_2(1+\sqrt3)$ and the silver exponent $\log_2(1+\sqrt2)$ to the factor $2$ in their local transfer estimate.
Their local Huber component has a fixed gradient direction.
It changes both adjacent coordinates at the same speed, limiting the room above the activation level for the next component.
We wondered whether we could increase this room by bending the local trajectory, gradually turning the descent direction toward coordinate $i+1$ so that the selected step raises it farther above its activation level.

We shared this intuition with ChatGPT-6 Astra.
Through several rounds of substantive interaction and detailed calculations by ChatGPT-6 Astra Ultra, we designed a smooth convex hard function whose local gradient rotates along a circular arc and whose coefficient of $s_i$ approaches $1$.
Combining this construction with a refinement of the recursive analysis in \citep[Section~3]{JungChoYun2026} gave the present result.

\bibliographystyle{alpha}
\bibliography{ref}

@article{Cauchy1847,
  author  = {Cauchy, Augustin-Louis},
  title   = {M{\'e}thode g{\'e}n{\'e}rale pour la r{\'e}solution des syst{\`e}mes d'{\'e}quations simultan{\'e}es},
  journal = {Comptes Rendus Hebdomadaires des S{\'e}ances de l'Acad{\'e}mie des Sciences},
  volume  = {25},
  pages   = {536--538},
  year    = {1847}
}

@article{Polyak1964,
  author  = {Polyak, Boris T.},
  title   = {Some Methods of Speeding Up the Convergence of Iteration Methods},
  journal = {USSR Computational Mathematics and Mathematical Physics},
  volume  = {4},
  number  = {5},
  pages   = {1--17},
  year    = {1964}
}

@article{Nesterov1983,
  author  = {Nesterov, Yurii E.},
  title   = {A Method of Solving a Convex Programming Problem with Convergence Rate {$O(1/k^2)$}},
  journal = {Soviet Mathematics Doklady},
  volume  = {27},
  number  = {2},
  pages   = {372--376},
  year    = {1983}
}

@book{Nesterov2004,
  author    = {Nesterov, Yurii},
  title     = {Introductory Lectures on Convex Optimization: A Basic Course},
  series    = {Applied Optimization},
  volume    = {87},
  publisher = {Kluwer Academic Publishers},
  year      = {2004}
}

@article{Huber1964,
  author  = {Huber, Peter J.},
  title   = {Robust Estimation of a Location Parameter},
  journal = {The Annals of Mathematical Statistics},
  volume  = {35},
  number  = {1},
  pages   = {73--101},
  year    = {1964}
}

@article{Moreau1965,
  author  = {Moreau, Jean-Jacques},
  title   = {Proximit{\'e} et dualit{\'e} dans un espace hilbertien},
  journal = {Bulletin de la Soci{\'e}t{\'e} Math{\'e}matique de France},
  volume  = {93},
  pages   = {273--299},
  year    = {1965},
  doi     = {10.24033/bsmf.1625}
}

@article{ParikhBoyd2014,
  author  = {Parikh, Neal and Boyd, Stephen},
  title   = {Proximal Algorithms},
  journal = {Foundations and Trends in Optimization},
  volume  = {1},
  number  = {3},
  pages   = {127--239},
  year    = {2014},
  doi     = {10.1561/2400000003}
}

@book{RockafellarWets1998,
  author    = {Rockafellar, R. Tyrrell and Wets, Roger J.-B.},
  title     = {Variational Analysis},
  series    = {Grundlehren der mathematischen Wissenschaften},
  volume    = {317},
  publisher = {Springer-Verlag},
  address   = {Berlin, Heidelberg},
  year      = {1998},
  doi       = {10.1007/978-3-642-02431-3}
}

@article{MaChen2026,
  author  = {Ma, Jianhao and Chen, Yuxin},
  title   = {A Lower Bound for Stepsize-Based Acceleration of Gradient Descent},
  journal = {arXiv preprint arXiv:2608.10418},
  year    = {2026}
}

@misc{Tsai2026,
  author       = {Tsai, Chung-En},
  title        = {An Improved Lower Bound for Non-Anytime Gradient Descent},
  year         = {2026},
  howpublished = {Blog post, \url{https://chungentsai.github.io/gd-lower-bounds.html}},
  url          = {https://chungentsai.github.io/gd-lower-bounds.html}
}

@article{YeLiu2026,
  author  = {Ye, Yuhan and Liu, Kaizhao},
  title   = {Improved Gradient Descent Lower Bounds Beyond {Nesterov}},
  journal = {arXiv preprint arXiv:2609.02855},
  year    = {2026}
}

@article{JungChoYun2026,
  author  = {Jung, Minchan and Cho, Hanseul and Yun, Chulhee},
  title   = {Stronger Lower Bounds for (Non-)Anytime Acceleration of Gradient Descent},
  journal = {arXiv preprint arXiv:2609.04032},
  year    = {2026}
}

@article{AltschulerParrilo2025,
  author        = {Altschuler, Jason M. and Parrilo, Pablo A.},
  title         = {Acceleration by Stepsize Hedging: {Silver Stepsize Schedule} for Smooth Convex Optimization},
  journal       = {Mathematical Programming},
  volume        = {213},
  number        = {1--2},
  pages         = {1105--1118},
  year          = {2025}
}

@article{GrimmerShuWang2025Composing,
  author        = {Grimmer, Benjamin and Shu, Kevin and Wang, Alex L.},
  title         = {Composing Optimized Stepsize Schedules for Gradient Descent},
  journal       = {Mathematics of Operations Research},
  year          = {2025},
  note          = {Articles in Advance},
  url           = {https://pubsonline.informs.org/doi/10.1287/moor.2024.0764}
}

@article{WangMaYangZhou2024,
  author  = {Wang, Bofan and Ma, Shiqian and Yang, Junfeng and Zhou, Danqing},
  title   = {Relaxed Proximal Point Algorithm: Tight Complexity Bounds and Acceleration Without Momentum},
  journal = {INFORMS Journal on Optimization},
  volume  = {8},
  number  = {2},
  pages   = {141--162},
  year    = {2026}
}

@article{TsaiFatkhullinZhangHe2026,
  author  = {Tsai, Chung-En and Fatkhullin, Ilyas and Zhang, Liang and He, Niao},
  title   = {Lower Bounds for Anytime Acceleration of Gradient Descent},
  journal = {arXiv preprint arXiv:2607.02053},
  year    = {2026}
}

@inproceedings{ZhangLeeDuChen2025,
  author    = {Zhang, Zihan and Lee, Jason D. and Du, Simon S. and Chen, Yuxin},
  title     = {Anytime Acceleration of Gradient Descent},
  booktitle = {Proceedings of Thirty Eighth Conference on Learning Theory},
  series    = {Proceedings of Machine Learning Research},
  volume    = {291},
  pages     = {5991--6013},
  year      = {2025},
  publisher = {PMLR}
}

@article{LevitinPolyak1966,
  author  = {Levitin, E. S. and Polyak, B. T.},
  title   = {Constrained Minimization Methods},
  journal = {USSR Computational Mathematics and Mathematical Physics},
  volume  = {6},
  number  = {5},
  pages   = {1--50},
  year    = {1966}
}

@article{Grimmer2024,
  author        = {Grimmer, Benjamin},
  title         = {Provably Faster Gradient Descent via Long Steps},
  journal       = {SIAM Journal on Optimization},
  volume        = {34},
  number        = {3},
  pages         = {2588--2608},
  year          = {2024}
}

@misc{GrimmerShuWang2023LongSteps,
  author = {Grimmer, Benjamin and Shu, Kevin and Wang, Alex L.},
  title = {Accelerated Gradient Descent via Long Steps},
  year = {2023},
  note = {Unpublished technical report, arXiv:2309.09961}
}

@article{GrimmerShuWang2025LongSteps,
  author        = {Grimmer, Benjamin and Shu, Kevin and Wang, Alex L.},
  title         = {Accelerated Objective Gap and Gradient Norm Convergence for Gradient Descent via Long Steps},
  journal       = {INFORMS Journal on Optimization},
  volume        = {7},
  number        = {2},
  pages         = {156--169},
  year          = {2025}
}

@mastersthesis{Altschuler2018,
  author = {Altschuler, Jason},
  title  = {Greed, Hedging, and Acceleration in Convex Optimization},
  school = {Massachusetts Institute of Technology},
  type   = {Master's thesis},
  year   = {2018}
}

@article{AltschulerParrilo2024RandomStepsizes,
  author        = {Altschuler, Jason M. and Parrilo, Pablo A.},
  title         = {Acceleration by Random Stepsizes: Hedging, Equalization, and the Arcsine Stepsize Schedule},
  journal       = {arXiv preprint arXiv:2412.05790},
  year          = {2024},
  doi           = {10.48550/arXiv.2412.05790},
  eprint        = {2412.05790},
  archiveprefix = {arXiv},
  note          = {To appear in Foundations of Computational Mathematics}
}

@article{AltschulerParrilo2025HedgingI,
  author    = {Altschuler, Jason M. and Parrilo, Pablo A.},
  title     = {Acceleration by Stepsize Hedging: Multi-Step Descent and the Silver Stepsize Schedule},
  journal   = {Journal of the ACM},
  volume    = {72},
  number    = {2},
  articleno = {12},
  pages     = {1--38},
  year      = {2025},
  publisher = {Association for Computing Machinery},
  doi       = {10.1145/3708502},
  note      = {Article 12}
}

@article{AltschulerParrilo2026Survey,
  author        = {Altschuler, Jason M. and Parrilo, Pablo A.},
  title         = {Stepsize Hedging: An Alternative Mechanism for Accelerating Gradient Descent},
  journal       = {INFORMS Computing Society Newsletter},
  pages         = {19--27},
  month         = jun,
  year          = {2026},
  eprint        = {2605.31386},
  archiveprefix = {arXiv},
  note          = {Research Highlight}
}

@article{BokAltschuler2026,
  author    = {Bok, Jinho and Altschuler, Jason M.},
  title     = {Optimized Methods for Composite Optimization: A Reduction Perspective},
  journal   = {Mathematical Programming},
  year      = {2026},
  doi       = {10.1007/s10107-026-02377-7}
}

@article{DasGuptaVanParysRyu2024,
  author        = {Das Gupta, Shuvomoy and Van Parys, Bart P. G. and Ryu, Ernest K.},
  title         = {{Branch-and-Bound Performance Estimation Programming}: A Unified Methodology for Constructing Optimal Optimization Methods},
  journal       = {Mathematical Programming},
  volume        = {204},
  number        = {1--2},
  pages         = {567--639},
  year          = {2024},
  doi           = {10.1007/s10107-023-01973-1},
  eprint        = {2203.07305},
  archiveprefix = {arXiv},
  primaryclass  = {math.OC}
}

@article{DroriTeboulle2014,
  author  = {Drori, Yoel and Teboulle, Marc},
  title   = {Performance of First-Order Methods for Smooth Convex Minimization: A Novel Approach},
  journal = {Mathematical Programming},
  volume  = {145},
  number  = {1--2},
  pages   = {451--482},
  year    = {2014},
  doi     = {10.1007/s10107-013-0653-0}
}

@article{Kim2024ExactGD,
  author = {Kim, Jungbin},
  title  = {A Proof of the Exact Convergence Rate of Gradient Descent},
  journal = {arXiv preprint arXiv:2412.04427},
  year = {2024}
}

@inproceedings{KornowskiShamir2024,
  author    = {Kornowski, Guy and Shamir, Ohad},
  title     = {Open Problem: Anytime Convergence Rate of Gradient Descent},
  booktitle = {Proceedings of the Thirty-Seventh Conference on Learning Theory},
  series    = {Proceedings of Machine Learning Research},
  volume    = {247},
  pages     = {5335--5339},
  year      = {2024},
  publisher = {PMLR}
}

@article{TaylorHendrickxGlineur2017,
  author        = {Taylor, Adrien B. and Hendrickx, Julien M. and Glineur, Fran{\c{c}}ois},
  title         = {Smooth Strongly Convex Interpolation and Exact Worst-Case Performance of First-Order Methods},
  journal       = {Mathematical Programming},
  volume        = {161},
  number        = {1--2},
  pages         = {307--345},
  year          = {2017},
  doi           = {10.1007/s10107-016-1009-3},
  eprint        = {1502.05666},
  archiveprefix = {arXiv},
  primaryclass  = {math.OC}
}

@article{Young1953,
  author  = {Young, David},
  title   = {On {Richardson's} Method for Solving Linear Systems with Positive Definite Matrices},
  journal = {Journal of Mathematics and Physics},
  volume  = {32},
  number  = {1--4},
  pages   = {243--255},
  year    = {1953},
  doi     = {10.1002/sapm1953321243}
}

@article{ZhangJiang2026,
  author  = {Zhang, Zehao and Jiang, Rujun},
  title   = {Accelerated Gradient Descent by Concatenation of Stepsize Schedules},
  journal = {SIAM Journal on Optimization},
  volume  = {36},
  number  = {2},
  pages   = {1182--1210},
  year    = {2026},
  doi     = {10.1137/25M173898X}
}

@article{KimFessler2016,
  author = {Kim, Donghwan and Fessler, Jeffrey A.},
  title = {Optimized First-Order Methods for Smooth Convex Minimization},
  journal = {Mathematical Programming},
  volume = {159},
  pages = {81--107},
  year = {2016}
}

@article{GrimmerJoPark2026,
  author = {Grimmer, Benjamin and Jo, Sunghyeon and Park, Chanwoo},
  title = {On Minimax Optimality and Uniqueness of Fixed-Step First-Order Methods for Smooth Convex Optimization},
  journal = {arXiv preprint arXiv:2609.07671},
  year = {2026}
}

@article{Drori2017,
  author = {Drori, Yoel},
  title = {The Exact Information-Based Complexity of Smooth Convex Minimization},
  journal = {Journal of Complexity},
  volume = {39},
  pages = {1--16},
  year = {2017}
}

@article{DroriTaylor2022,
  author  = {Drori, Yoel and Taylor, Adrien B.},
  title   = {On the Oracle Complexity of Smooth Strongly Convex Minimization},
  journal = {Journal of Complexity},
  volume  = {68},
  pages   = {101590},
  year    = {2022}
}

@article{DroriTaylor2020,
  author = {Drori, Yoel and Taylor, Adrien B.},
  title = {Efficient First-Order Methods for Convex Minimization: A Constructive Approach},
  journal = {Mathematical Programming},
  volume = {184},
  pages = {183--220},
  year = {2020}
}

@article{GrimmerShuWang2026Survey,
  author = {Grimmer, Benjamin and Shu, Kevin and Wang, Alex L.},
  title = {Subgame Perfection and Dynamic Optimality in Optimization},
  journal = {SIAG on Optimization Views and News},
  volume = {34},
  number = {1},
  pages = {2--8},
  year = {2026}
}

@article{Bubeck2015,
  author  = {Bubeck, S{\'e}bastien},
  title   = {Convex Optimization: Algorithms and Complexity},
  journal = {Foundations and Trends in Machine Learning},
  volume  = {8},
  number  = {3--4},
  pages   = {231--357},
  year    = {2015}
}

@inproceedings{CaiWuMeiLindseyBartlett2024,
  author = {Cai, Yuhang and Wu, Jingfeng and Mei, Song and Lindsey, Michael and Bartlett, Peter L.},
  title = {Large Stepsize Gradient Descent for Non-Homogeneous Two-Layer Networks: Margin Improvement and Fast Optimization},
  booktitle = {Advances in Neural Information Processing Systems},
  volume = {37},
  year = {2024}
}

@article{Gavurin1950,
  author = {Gavurin, Mark K.},
  title = {The Use of Polynomials of Best Approximation for Improving the Convergence of Iterative Processes},
  journal = {Uspekhi Matematicheskikh Nauk},
  volume = {5},
  number = {3(37)},
  pages = {156--160},
  year = {1950},
  note = {In Russian. English translation in NBS Report 2007 (1952), pp. 44--50}
}

@article{GolubVarga1961,
  author  = {Golub, Gene H. and Varga, Richard S.},
  title   = {{Chebyshev} Semi-Iterative Methods, Successive Overrelaxation Iterative Methods, and Second Order {Richardson} Iterative Methods. {Part I}},
  journal = {Numerische Mathematik},
  volume  = {3},
  pages   = {147--156},
  year    = {1961}
}

@article{HestenesStiefel1952,
  author = {Hestenes, Magnus R. and Stiefel, Eduard},
  title = {Methods of Conjugate Gradients for Solving Linear Systems},
  journal = {Journal of Research of the National Bureau of Standards},
  volume = {49},
  number = {6},
  pages = {409--436},
  year = {1952}
}

@article{KaleManupriyaMarionBachRaj2026,
  author = {Kale, Sacchit and Manupriya, Piyushi and Marion, Pierre and Bach, Francis and Raj, Anant},
  title = {Stretched Exponential Convergence of (Stochastic) Gradient Descent for Separable Logistic Regression},
  journal = {Transactions on Machine Learning Research},
  year = {2026}
}

@book{NemirovskyYudin1979,
  author = {Nemirovsky, Arkadii S. and Yudin, David B.},
  title = {Slozhnost' Zadach i Effektivnost' Metodov Optimizatsii},
  publisher = {Nauka},
  address = {Moscow},
  year = {1979},
  note = {In Russian. English translation: Problem Complexity and Method Efficiency in Optimization, Wiley, 1983}
}

@article{Oymak2021,
  author = {Oymak, Samet},
  title = {Provable Super-Convergence With a Large Cyclical Learning Rate},
  journal = {IEEE Signal Processing Letters},
  volume = {28},
  pages = {1645--1649},
  year = {2021}
}

@book{Polyak1987,
  author    = {Polyak, Boris T.},
  title     = {Introduction to Optimization},
  publisher = {Optimization Software, Inc., Publications Division},
  address   = {New York},
  year      = {1987},
  isbn      = {0911575146}
}

@inproceedings{WuBartlettTelgarskyYu2024,
  author = {Wu, Jingfeng and Bartlett, Peter L. and Telgarsky, Matus and Yu, Bin},
  title = {Large Stepsize Gradient Descent for Logistic Loss: Non-Monotonicity of the Loss Improves Optimization Efficiency},
  booktitle = {Proceedings of the Thirty-Seventh Conference on Learning Theory},
  series = {Proceedings of Machine Learning Research},
  volume = {247},
  pages = {5019--5073},
  publisher = {PMLR},
  year = {2024}
}

@inproceedings{WuMarionBartlett2025,
  author = {Wu, Jingfeng and Marion, Pierre and Bartlett, Peter L.},
  title = {Large Stepsizes Accelerate Gradient Descent for Regularized Logistic Regression},
  booktitle = {Advances in Neural Information Processing Systems},
  volume = {38},
  year = {2025}
}

@article{WuZhengCaiLuo2026,
  author = {Wu, Yue and Zheng, Weiqiang and Cai, Yang and Luo, Haipeng},
  title = {Accelerating Min-Max Optimization via Power-Law Stepsizes},
  journal = {arXiv preprint arXiv:2606.01764},
  year = {2026},
  note = {Accepted at {NeurIPS} 2026}
}

@article{LiuChenJiangWang2026,
  author = {Liu, Yu and Chen, Kang and Jiang, Rujun and Wang, Tianyu},
  title = {Optimal Recursive Composition and Dyadic Phase Laws for Gradient Descent with Predetermined Stepsizes},
  journal = {arXiv preprint arXiv:2609.11788},
  year = {2026}
}

@article{Topkis1978,
  author = {Topkis, Donald M.},
  title = {Minimizing a Submodular Function on a Lattice},
  journal = {Operations Research},
  volume = {26},
  number = {2},
  pages = {305--321},
  year = {1978}
}

\clearpage
\appendix
\addtocontents{toc}{\protect\setcounter{tocdepth}{1}}
\crefalias{section}{appendix}
\section{Proofs for the Hard-Function Construction}
\label{app:local}
We use the parameters $\epsilon,\gamma,c,R$ from \cref{thm:transfer}.

\subsection{Proof of Lemma~\ref{lem:local-backward}}
\label{local:app-backward}
\begin{proof}
\label{local:backward-proof}

We first construct the gap iterates backward, then verify the projection identity $v_j=\Pi_K(r_j)$ and bound the displacements $P$ and $Q$.

\paragraph{Step 1. Constructing the gap trajectory.}
Set
\begin{equation}
  v(q):=(p(q),-q),\qquad
  p(q):=\sqrt{R^2-(q+\gamma)^2},\qquad
  0\le q\le1.
  \label{local:arc}
\end{equation}
The vectors $v(q)$ lie on the circle of radius $R$ centered at $(0,\gamma)$, and
\begin{equation}
  v(0)=(c,0),\qquad v(1)=(\epsilon,-1),\qquad
  \epsilon\le p(q)\le c,\qquad \frac{q}{p(q)}\le\frac1\epsilon.
  \label{local:arc-bounds}
\end{equation}
Let
\[
  K:=\operatorname{conv}\bigl(\{0\}\cup\{v(q):0\le q\le1\}\bigr).
\]
The set $K$ lies in the closed disk centered at $(0,\gamma)$ with radius $R$, and every $w\in K$ satisfies $w^{(1)}\le c$.
Fix stepsizes $\alpha_1,\ldots,\alpha_m\ge0$ and set $S=\sum_{j=1}^m\alpha_j$.
Write $r_j=(u_j,-w_j)$ and $v_j=(p_j,-q_j)$ for $0\le j\le m$.
Start with $r_m=v_m=(\epsilon,-1)$.
For $j=m,m-1,\ldots,1$, define
\[
  \Delta_j:=w_j+\gamma-\alpha_j\gamma.
\]
If $\Delta_j\le u_j\gamma/c$, choose $v_{j-1}=(p_{j-1},-q_{j-1}):=(c,0)$.
Otherwise, choose
\begin{equation}
  d_j:=\sqrt{u_j^2+\Delta_j^2},\qquad
  q_{j-1}:=\frac{R\Delta_j}{d_j}-\gamma,\qquad
  p_{j-1}:=p(q_{j-1})=\frac{Ru_j}{d_j}.
  \label{local:backward-direction}
\end{equation}
In both cases, set $r_{j-1}:=r_j+\alpha_jv_{j-1}$, so
\[
  u_{j-1}=u_j+\alpha_jp_{j-1},\qquad
  w_{j-1}=w_j+\alpha_jq_{j-1}.
\]
Geometrically, when $\Delta_j>u_j\gamma/c$, we choose $v_{j-1}$ where the ray from $(0,\gamma)$ through $r_j+\alpha_j(0,\gamma)$ meets the circle, as shown in \cref{fig:backward-similar-triangles}.
If the ray passes at or above the corner $(c,0)$, we choose the corner instead.
\begin{figure}[H]
  \centering
  \includegraphics[width=\linewidth]{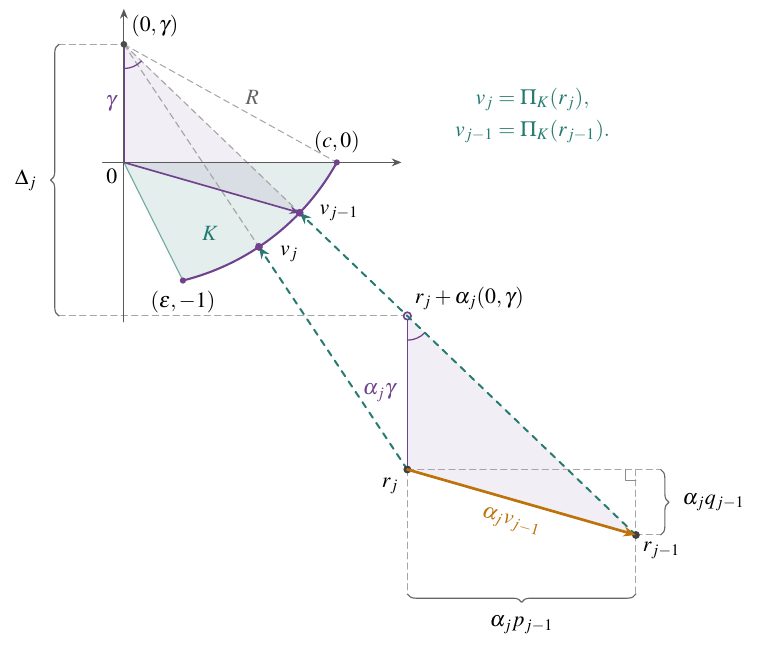}
  \caption{The backward construction in \eqref{local:backward-direction}. Pale green denotes $K$, and green dashed arrows show $v_j=\Pi_K(r_j)$ and $v_{j-1}=\Pi_K(r_{j-1})$. The purple triangles are similar, with scale factor $\alpha_j$.}
  \label{fig:backward-similar-triangles}
\end{figure}

\paragraph{Step 2. Verifying the projection identity.}
We prove by backward induction that, for $0\le j\le m$,
\begin{equation}
  v_j=\Pi_K(r_j),\qquad
  u_j\ge\epsilon,\qquad w_j\ge1,\qquad \frac{w_j}{u_j}\le\frac1\epsilon.
  \label{local:projection-invariants}
\end{equation}
The statements in \eqref{local:projection-invariants} hold at $r_m=v_m=(\epsilon,-1)$. Assume they hold at $r_j$. We first check that the chosen vector lies on the arc and that the coordinate bounds hold at $r_{j-1}$.
If $\Delta_j>u_j\gamma/c$, then
\[
  \frac\gamma c<\frac{\Delta_j}{u_j}
  \le\frac{w_j+\gamma}{u_j}
  \le\frac{1+\gamma}{\epsilon}.
\]
By \eqref{local:backward-direction}, $q_{j-1}=R(\Delta_j/u_j)/\sqrt{1+(\Delta_j/u_j)^2}-\gamma$.
The map $s\mapsto Rs/\sqrt{1+s^2}-\gamma$ is increasing and takes the values $0$ and $1$ at $s=\gamma/c$ and $s=(1+\gamma)/\epsilon$, respectively.
Hence $q_{j-1}\in(0,1]$.
If $\Delta_j\le u_j\gamma/c$, then $(p_{j-1},q_{j-1})=(c,0)$.
In both cases, $p_{j-1}\ge\epsilon$ and $q_{j-1}\ge0$, giving $u_{j-1}\ge\epsilon$ and $w_{j-1}\ge1$.
The inductive ratio bound and \eqref{local:arc-bounds} also give
\[
  w_{j-1}=w_j+\alpha_jq_{j-1}
  \le\frac{u_j+\alpha_jp_{j-1}}{\epsilon}
  =\frac{u_{j-1}}{\epsilon}.
\]

It remains to verify the projection. For $v\in K$, the identity $v=\Pi_K(r)$ is equivalent to $\langle r-v,w-v\rangle\le0$ for every $w\in K$.
If $\Delta_j>u_j\gamma/c$, we show that $r_{j-1}-v_{j-1}$ points along the outward radius at $v_{j-1}$.
By \eqref{local:backward-direction},
\[
  r_j+\alpha_j(0,\gamma)-(0,\gamma)
  =(u_j,-\Delta_j)
  =\frac{d_j}{R}\bigl(v_{j-1}-(0,\gamma)\bigr).
\]
Substituting this into $r_{j-1}=r_j+\alpha_jv_{j-1}$ gives
\begin{equation}
  r_{j-1}-v_{j-1}
  =\left(\frac{d_j}{R}+\alpha_j-1\right)
    \bigl(v_{j-1}-(0,\gamma)\bigr).
  \label{local:curved-normal}
\end{equation}
The reverse triangle inequality and the coordinate bounds give
\[
  d_j
  \ge\sqrt{u_j^2+(w_j+\gamma)^2}-\alpha_j\gamma
  \ge R-\alpha_j\gamma.
\]
Thus the coefficient in \eqref{local:curved-normal} is nonnegative, since
\[
  \frac{d_j}{R}+\alpha_j-1
  \ge\alpha_j\left(1-\frac\gamma R\right)\ge0.
\]
Because $K$ lies inside the disk, every $w\in K$ satisfies
\[
  \langle v_{j-1}-(0,\gamma),w-v_{j-1}\rangle
  \le R\|w-(0,\gamma)\|-R^2\le0.
\]
This proves $v_{j-1}=\Pi_K(r_{j-1})$.

If $\Delta_j\le u_j\gamma/c$, then $v_{j-1}=(c,0)$ and $w_{j-1}=w_j$. Rearranging this condition gives
\[
  u_{j-1}-c=u_j+\alpha_jc-c\ge\frac{cw_j}{\gamma}.
\]
For $w\in K$, we have $w^{(1)}\le c$ and $\langle(c,-\gamma),w-(c,0)\rangle\le0$, since $(c,-\gamma)$ is the outward radius at $(c,0)$. Hence
\begin{equation}
\begin{aligned}
  \langle r_{j-1}-(c,0),w-(c,0)\rangle
  &=(u_{j-1}-c)(w^{(1)}-c)-w_jw^{(2)}\\
  &\le\frac{w_j}{\gamma}\langle(c,-\gamma),w-(c,0)\rangle\le0.
\end{aligned}
  \label{local:horizontal-normal}
\end{equation}
By \eqref{local:horizontal-normal}, $v_{j-1}=\Pi_K(r_{j-1})$, completing the induction.

\paragraph{Step 3. Bounding $P$ and $Q$.}
\label{local:turning-proof}
We first identify the indices at which $q_{j-1}=0$.
If $\Delta_j\le u_j\gamma/c$ and $j\ge2$, the choice $v_{j-1}=(c,0)$ gives $w_{j-1}=w_j$ and $u_{j-1}=u_j+\alpha_jc$. Hence
\[
  \Delta_{j-1}
  =w_{j-1}+\gamma-\alpha_{j-1}\gamma
  \le w_j+\gamma
  \le\frac{u_j\gamma}{c}+\alpha_j\gamma
  =\frac{u_{j-1}\gamma}{c}.
\]
Thus the backward recursion also chooses $v_{j-2}=(c,0)$.
Repeating this argument shows that once $(c,0)$ is chosen, all earlier gradient vectors equal $(c,0)$. Thus there is an index $j_\star\in\{0,\ldots,m\}$ such that, for $1\le j\le m$,
\begin{equation}\label{eq:j-star}
    \Delta_j>\frac{u_j\gamma}{c}
  \quad\Longleftrightarrow\quad
  j_\star<j\le m,
\end{equation}
where $j_\star=m$ if the strict inequality never holds.

The coordinate recurrences give
\begin{equation*}
  P=\sum_{j=1}^m\alpha_jp_{j-1},\qquad
  Q=\sum_{j=1}^m\alpha_jq_{j-1}.
\end{equation*}
Since $p_{j-1}\le c$, we have $P\le cS$. To bound $Q$ independently of $S$, use $q_{j-1}=0$ for $j\le j_\star$ and $q_{j-1}\le p_{j-1}/\epsilon$ from \eqref{local:arc-bounds}. These give
\[
  0\le Q
  =\sum_{j=j_\star+1}^m\alpha_jq_{j-1}
  \le\frac1\epsilon\sum_{j=j_\star+1}^m\alpha_jp_{j-1}
  =\frac{u_{j_\star}}\epsilon-1.
\]
The last equality follows from $\alpha_jp_{j-1}=u_{j-1}-u_j$ and $u_m=\epsilon$.
It therefore suffices to bound $u_{j_\star}/\epsilon$.

Set $\sigma_j:=(w_j+\gamma)/u_j$, the magnitude of the slope from the circle center $(0,\gamma)$ to $r_j=(u_j,-w_j)$. For $j_\star<j\le m$, the recurrence and \eqref{local:backward-direction} give
\[
  (u_{j-1},w_{j-1}+\gamma)
  =\left(1+\frac{\alpha_jR}{d_j}\right)(u_j,\Delta_j),
\]
so
\[
  \sigma_{j-1}=\frac{\Delta_j}{u_j}
  =\sigma_j-\frac{\alpha_j\gamma}{u_j},
  \qquad
  p_{j-1}=\frac{R}{\sqrt{1+\sigma_{j-1}^2}}.
\]
Thus the change in slope determines the increase in $u_j$ under the backward recursion,
\begin{equation}
  \frac{u_{j-1}}{u_j}
  =1+\frac{\alpha_j}{u_j}p_{j-1}
  =1+\frac R\gamma
    \frac{\sigma_j-\sigma_{j-1}}
         {\sqrt{1+\sigma_{j-1}^2}}.
  \label{local:turning-recursion}
\end{equation}
To obtain an inequality that telescopes when multiplied over $j$, set $A(\sigma):=\sigma+\sqrt{1+\sigma^2}$. Its derivative satisfies $A'(\sigma)/A(\sigma)=1/\sqrt{1+\sigma^2}$, which matches the denominator in \eqref{local:turning-recursion}. Convexity of $A$ gives
\[
  \frac{A(\sigma_j)}{A(\sigma_{j-1})}
  \ge1+\frac{\sigma_j-\sigma_{j-1}}{\sqrt{1+\sigma_{j-1}^2}}.
\]
Since $\sigma_j\ge\sigma_{j-1}$ and $R/\gamma>1$, Bernoulli's inequality gives
\begin{equation}
  \frac{u_{j-1}}{u_j}
  \le\left(1+\frac{\sigma_j-\sigma_{j-1}}{\sqrt{1+\sigma_{j-1}^2}}\right)^{R/\gamma}
  \le\left(\frac{A(\sigma_j)}{A(\sigma_{j-1})}\right)^{R/\gamma}.
  \label{local:turning-telescope}
\end{equation}
Multiplying \eqref{local:turning-telescope} over $j=j_\star+1,\ldots,m$ cancels all intermediate terms, leaving
\[
  \frac{u_{j_\star}}{\epsilon}
  \le\left[\frac{A(\sigma_m)}{A(\sigma_{j_\star})}\right]^{R/\gamma}.
\]
Here $u_m=\epsilon$ and $\sigma_m=(1+\gamma)/\epsilon$.
If $j_\star<m$, then $\sigma_{j_\star}=\Delta_{j_\star+1}/u_{j_\star+1}>\gamma/c$ by \eqref{eq:j-star}; if $j_\star=m$, then $\sigma_{j_\star}=\sigma_m\ge\gamma/c$.
Since $A$ is increasing,
\[
  A(\sigma_m)=\frac{1+\gamma+R}{\epsilon},
  \qquad
  A(\sigma_{j_\star})\ge A(\gamma/c)=\frac{\gamma+R}{c}.
\]
Combining these estimates yields
\begin{equation}
  0\le Q\le\frac{u_{j_\star}}\epsilon-1
  \le H_0:=\left[\frac{c(1+\gamma+R)}{\epsilon(\gamma+R)}\right]^{R/\gamma}-1.
  \label{local:choice-H0}
\end{equation}
\end{proof}

\subsection{Proof of Lemma~\ref{lem:local-transfer}}
\label{local:app-transfer}

We use the following properties of Moreau envelopes in the proof of \cref{lem:local-transfer}.

\begin{fact}[Moreau envelopes]
\label{lem:moreau-envelope}
Let $g:\R^d\to(-\infty,+\infty]$ be proper, closed, and convex.
Its Moreau envelope with parameter $1$ and its proximal map are
\[
  \operatorname{env}_1 g(x)
  :=\min_{y\in\R^d}\left\{g(y)+\frac12\|x-y\|^2\right\},
  \qquad
  \operatorname{prox}_g(x)
  :=\argminop_{y\in\R^d}\left\{g(y)+\frac12\|x-y\|^2\right\}.
\]
The function $\operatorname{env}_1 g$ is finite, convex, and continuously differentiable, and its gradient is $1$-Lipschitz and satisfies $\nabla\operatorname{env}_1 g(x)=x-\operatorname{prox}_g(x)$~\citep[Propositions~5.b, 7.b, and~7.d]{Moreau1965}.

Moreover, if $K\subset\R^d$ is nonempty, closed, and convex and $\sigma_K(x):=\sup_{v\in K}\langle v,x\rangle$ is its support function, then
\begin{equation*}
  \nabla\operatorname{env}_1\sigma_K(x)=\Pi_K(x),
\end{equation*}
where $\Pi_K$ denotes Euclidean projection onto $K$.
If $g\ge0$, then $\operatorname{env}_1g\ge0$ and the two functions have the same zero set.
For background, see~\citep[Sections~1.G and~2.D]{RockafellarWets1998} and~\citep[Sections~3.1 and~6.4.2]{ParikhBoyd2014}.
\end{fact}

\begin{proof}[Proof of \cref{lem:local-transfer}]
Apply \cref{lem:local-backward} with $\alpha_j=h_j/2$, so $S=s/2$.
Using $H_0$ defined in \eqref{local:choice-H0}, set
\begin{equation*}
  C_0:=\epsilon+\frac{H_0+1}{\epsilon}.
\end{equation*}
The resulting initial point and displacements satisfy
\begin{equation}\label{eq:cSplusC0}
  r_0=(\epsilon+P,-1-Q),\qquad
  \epsilon+P+\frac{1+Q}{\epsilon}\le\frac{cs}{2}+C_0.
\end{equation}

\paragraph{Step 1. Defining the local hard function.}
Set $\rho=1/[c(a+s)]$ and $z_\circ=(1,0)-2\rho r_0$.
Then the scaled trajectory $z_\circ+2\rho r_j$ starts at $(1,0)$.
Let $\mathcal V=\{(0,0),v_0,\ldots,v_m,(c,0)\}$ and define
\begin{equation}
  \phi(z):=2\rho\max_{v\in\mathcal V}\langle v,z-z_\circ\rangle,
  \qquad f:=\operatorname{env}_1\phi.
  \label{local:finite-envelope}
\end{equation}
Since $0\in\mathcal V$, we have $\phi\ge0$, so \cref{lem:moreau-envelope} shows that $f$ is nonnegative, convex, and $1$-smooth.
Let $K_{\mathcal V}=\operatorname{conv}(\mathcal V)$. Since $v_j=\Pi_K(r_j)\in K_{\mathcal V}\subseteq K$, we also have $\Pi_{K_{\mathcal V}}(r_j)=v_j$.
Also, $\phi(z)=\sigma_{2\rho K_{\mathcal V}}(z-z_\circ)$.
The projection formula in \cref{lem:moreau-envelope} gives $\nabla f(z)=\Pi_{2\rho K_{\mathcal V}}(z-z_\circ)$ and hence
\begin{equation}
  \nabla f(z_\circ+2\rho r_j)=2\rho v_j\quad(0\le j\le m).
  \label{local:scaled-gradients}
\end{equation}
Write $(x_j,y_j):=z_\circ+2\rho r_j$ for the scaled and translated points. Since $r_j=r_{j-1}-(h_j/2)v_{j-1}$, \eqref{local:scaled-gradients} shows that GD on $f/2$ with stepsizes $h_1,\ldots,h_m$ visits $(x_j,y_j)$ for $0\le j\le m$.
Their first coordinate is nonincreasing and at least
\[
  1-2\rho P\ge1-\frac{s}{a+s}>0,
\]
and their second coordinate is nondecreasing from $0$ to $2\rho Q$.

\paragraph{Step 2. Verifying the zero region of $f$ and $\partial_2f$.}
We next verify \eqref{local:zero-regions}, which gives the transfer property of $\Phi_i$ after the scaling and translation in \eqref{local:chain-component}.
For $v=(p,-q)\in\mathcal V$, we have $p,q\ge0$ and $q\le p/\epsilon$.
The definition of $a$ in \cref{thm:transfer} gives $ca\ge2C_0$, so \eqref{eq:cSplusC0} implies
\[
  \langle v,-z_\circ\rangle
  =-p+2\rho\bigl[p(\epsilon+P)+q(1+Q)\bigr]
  \le-p+\rho p(cs+2C_0)\le0.
\]
When $z_1\le0$ and $z_2\ge0$, we also have $\langle v,z\rangle=pz_1-qz_2\le0$, so every affine piece $2\rho\langle v,z-z_\circ\rangle$ of $\phi$ is nonpositive.
Thus $\phi=0$ in this region, and \cref{lem:moreau-envelope} gives $f=0$ there.
Since $f\ge0$, its gradient is also zero there.

Set $\tau=z_\circ^{(2)}=2\rho(Q+1)>0$, so $0\le y_j\le2\rho Q=\tau-2\rho$ for $0\le j\le m$.
For every $(w_1,-w_2)\in2\rho K_{\mathcal V}$, the point $(w_1,0)$ also belongs to $2\rho K_{\mathcal V}$, since $0\le w_1\le2\rho c$ and $[0,2\rho c]\times\{0\}\subset2\rho K_{\mathcal V}$.
If $z=(x,y)$ with $y\ge\tau$ and $w_2>0$, then
\[
  \|z-z_\circ-(w_1,-w_2)\|^2
  -\|z-z_\circ-(w_1,0)\|^2
  =2(y-\tau)w_2+w_2^2>0.
\]
Thus any point of $2\rho K_{\mathcal V}$ with negative second coordinate is farther from $z-z_\circ$ than a point with second coordinate zero. The projection, and hence $\nabla f(x,y)$, must have second coordinate zero, proving
\begin{equation}
  \partial_2 f(x,y)=0\qquad(x\in\R,\ y\ge\tau).
  \label{local:output-isolation}
\end{equation}

\paragraph{Step 3. The checkpoint update.}
Just before the checkpoint update,
\[
  (x_m,y_m)=z_\circ+2\rho r_m=(1-2\rho P,\tau-2\rho),
  \qquad \nabla(f/2)(x_m,y_m)=\rho(\epsilon,-1).
\]
The step of size $b$ on $f/2$ gives
\[
  y_{m+1}=y_m+b\rho=\tau+\rho(b-2).
\]
The second coordinate therefore exceeds $\tau$ by $\rho(b-2)>0$.
\end{proof}

\subsection[Smoothness and the final function value]{Smoothness and the final function value}
\label{local:app-chain}

For the objective $F$ in \eqref{local:global-objective}, nonnegativity and \eqref{local:zero-regions} give $F(0)=\min F=0$, since every $\ell_i\ge0$.
For $d=y-x$, the component smoothness inequalities give
\begin{align*}
  F(y)&\le F(x)+\langle\nabla F(x),d\rangle
    +\frac14\left[
      \sum_{i=1}^k\bigl((d^{(i)})^2+(d^{(i+1)})^2\bigr)
      +(d^{(k+1)})^2\right]\\*
  &\le F(x)+\langle\nabla F(x),d\rangle+\frac12\|d\|^2.
\end{align*}
Each coordinate appears at most twice in the bracketed sum, so $F$ is $1$-smooth.

For the final function value, use $\delta=D_{k+1}/(1+s_{k+1})$ as chosen in \cref{sec:local}.
For $t_k\le t<n$, if $x_t^{(k+1)}-\ell_{k+1}\ge\delta$, the Huber term is affine and the transfer property \eqref{eq:transfer-property} gives $\partial_{x^{(k+1)}}\Phi_k(x_t)=0$ when $k>0$.
Thus the last coordinate decreases by $h_{t+1}\delta/2$.
Starting from $x_{t_k}^{(k+1)}-\ell_{k+1}=D_{k+1}\ge\delta$, induction shows that the final iterates remain in this affine region and satisfy
\[
\begin{aligned}
  x_t^{(k+1)}-\ell_{k+1}
  &=D_{k+1}-\frac{\delta}{2}\sum_{t'=t_k+1}^{t}h_{t'}\\
  &\ge D_{k+1}-\frac{\delta s_{k+1}}2
   =\delta\left(1+\frac{s_{k+1}}2\right)\ge\delta
   \qquad(t_k\le t\le n).
\end{aligned}
\]
All other components are nonnegative, so
\[
  F(x_n)\ge\frac12\left[
    \delta\left(D_{k+1}-\frac{\delta s_{k+1}}2\right)-\frac{\delta^2}2
  \right]
  =\frac{D_{k+1}^2}{4(1+s_{k+1})}.
\]
The same calculation includes $k=0$, with $t_0=\ell_1=0$ and $D_1=1$.

\section{Proofs for the Non-anytime Lower Bound}
\label{app:global}

\paragraph{Notation.} We use the constants and conventions from \cref{sec:global}, with
$0<\epsilon<1/2$, $c=\sqrt{1+2\gamma+\epsilon^2}$,
$a$ as in \cref{thm:transfer}, and $\Lambda\ge a$.
For an empty subsequence, set $V_\mu(\varnothing)=c\mu$.
A minimum over no indices is $+\infty$.
An inadmissible checkpoint set has
$Q_\Lambda(T;H)=+\infty$.

\subsection[Proof of (\ref{global:eq-cost-lower-bound})]{Proof of \eqref{global:eq-cost-lower-bound}}
\label{global:app-cost-lower-bound}

Applying the transfer bound \eqref{eq:transfer} to $T$ gives
\[
  \mathcal R_n(H)\ge
  \frac{1}{4(1+s_{k+1})}
  \prod_{i=1}^k
  \left[\frac{b_i-2}{c(a+s_i)}\right]^2
  =\frac{c^2(\Lambda+s_{k+1})^2}
        {4(1+s_{k+1})Q_\Lambda(T;H)^2}.
\]
Finally, the inequality
\[
 (\Lambda+s)^2-4(\Lambda-1)(1+s)=(s-\Lambda+2)^2\ge0
\]
with $s=s_{k+1}$ proves the claim.

\subsection{Proof of Lemma~\ref{global:lem-extremal-schedule}}
The recursion \eqref{global:eq-value-decomposition} describes $V_\Lambda(H)$ for a fixed schedule $H$, whereas we need to bound its largest possible value $U_n(\Lambda)$ over all schedules of length $n$.
\Cref{global:lem-extremal-schedule} gives a maximizing schedule $H$ at which both the empty set and a singleton $\{t_*\}$ are tight.
Splitting at $t_*$ gives $U_n(\Lambda)=V_a(H_{1:t_*-1})V_\Lambda(H_{t_*+1:n})/(h_{t_*}-2)$.
The two numerator factors concern shorter schedules and are bounded by $U_{t_*-1}(a)$ and $U_{n-t_*}(\Lambda)$.
Together with tightness of the empty set, these bounds let us apply induction on $n$, as carried out in \cref{global:app-uniform-growth}.

To prove the lemma, we first show that $\log Q_\Lambda(T;H)$ is submodular in the checkpoint set $T$, following~\citep[Lemma~C.2]{JungChoYun2026}.

\begin{lemma}[Log-submodularity]
\label{global:lem-log-submodularity}
For any schedule $H\in[0,\infty)^n$ and any two admissible checkpoint sets
$S,T$,
\begin{equation}
  Q_\Lambda(S;H)Q_\Lambda(T;H)
  \ge Q_\Lambda(S\cap T;H)Q_\Lambda(S\cup T;H).
  \label{global:eq-log-submodularity}
\end{equation}
If $S$ and $T$ are disjoint and nonempty, the inequality is strict.
Consequently, tight sets are closed under unions and intersections, and two
nonempty tight sets cannot be disjoint.
\end{lemma}

\begin{proof}
\emph{Inserting a checkpoint.}
Insert a checkpoint of stepsize $b>2$ into a gap, and let $\ell,r\ge0$ be the sums of stepsizes before and after it within that gap, respectively.
Put
\[
 K:=\begin{cases}
       a+r,&\text{if the gap ends at an existing checkpoint},\\
       \Lambda+r,&\text{if it is the final gap}.
     \end{cases}
\]
Before insertion, this gap contributes $c(\ell+b+K)$ to $Q_\Lambda$. After insertion, it contributes $c(a+\ell)cK/(b-2)$, so the cost is multiplied by
\[
 \frac{c(a+\ell)K}{(b-2)(\ell+b+K)}.
\]
Its logarithmic derivatives with respect to $\ell$ and $K$ are $1/(a+\ell)-1/(\ell+b+K)$ and $1/K-1/(\ell+b+K)$, respectively. Both are positive because $b>2$ and $K\ge a$, so the ratio is strictly increasing in both variables.

\smallskip
\noindent\emph{Comparing checkpoint sets.}
Insert the checkpoints in $S\setminus T$, one at a time in the same order, into $S\cap T$ and into $T$.
At each insertion, the set built from $T$ contains the set built from $S\cap T$. Its gap containing the new checkpoint therefore starts no earlier and ends no later, so its values of $\ell$ and $K$ are no larger.
If this gap ends at a checkpoint while the corresponding gap for the set built from $S\cap T$ is final, the constant in $K$ also decreases from $\Lambda$ to $a\le\Lambda$.
Multiplying the insertion ratios gives
\[
 \frac{Q_\Lambda(S;H)}{Q_\Lambda(S\cap T;H)}
 \ge
 \frac{Q_\Lambda(S\cup T;H)}{Q_\Lambda(T;H)},
\]
which proves \eqref{global:eq-log-submodularity}.
If $S$ and $T$ are disjoint and nonempty, the first step inserts the same checkpoint into the empty set and into $T$.
Since every checkpoint already in $T$ has stepsize greater than $2$, at least one of $\ell$ and $K$ is strictly smaller when inserting into $T$.
The first comparison, and hence the final inequality, is therefore strict.

Finally, suppose that $S$ and $T$ are tight, so both attain $V:=V_\Lambda(H)$.
Every checkpoint set has value at least $V$, and \eqref{global:eq-log-submodularity} gives
\[
 V^2\ge
 Q_\Lambda(S\cap T;H)Q_\Lambda(S\cup T;H)\ge V^2.
\]
Both factors in the middle must therefore equal $V$, so $S\cap T$ and $S\cup T$ are tight.\footnote{Closure under unions and intersections is a standard property of submodular minimizers~\citep{Topkis1978}, applied here to $\log Q_\Lambda(\cdot;H)$.}
If $S$ and $T$ were nonempty and disjoint, the first inequality would be strict, giving $V^2>V^2$, a contradiction.
\end{proof}

We now prove \cref{global:lem-extremal-schedule}.
After establishing existence of a maximizing schedule, we show that it has a nonempty tight checkpoint set.
By \cref{global:lem-log-submodularity}, the intersection of all such sets is again nonempty and tight.
We then perturb the stepsizes at checkpoints in this intersection to show that it contains only one checkpoint, following~\citep[Appendix~C.6]{JungChoYun2026}.

\begin{proof}[Proof of \cref{global:lem-extremal-schedule}]

Appending zero steps does not change $V_\Lambda$, so $U_m(\Lambda)\le U_j(\Lambda)$ whenever $m\le j$.

\paragraph{Step 1. Finiteness of $U_n(\Lambda)$.}
We induct on $n$, starting from $U_0(\Lambda)=c\Lambda$.
For an arbitrary schedule $H$, let $M=\max_t h_t$.
If $M\le2a+4$, the empty checkpoint set gives
\[
 V_\Lambda(H)\le c\bigl[\Lambda+n(2a+4)\bigr].
\]
If $M>2a+4$, choose an index $j$ with $h_j=M$ and split at $j$ in \eqref{global:eq-value-decomposition}.
Using the empty set on the left, $M-2\ge M/2$, and $a<M$, we obtain
\[
 V_\Lambda(H)
 \le\frac{c\bigl(a+\sum_{t<j}h_t\bigr)}{M-2}
       U_{n-j}(\Lambda)
 \le2cn\,U_{n-1}(\Lambda).
\]
Both bounds are independent of $H$, and $U_{n-1}(\Lambda)$ is finite by the induction hypothesis.
Thus $U_n(\Lambda)$ is finite.

\paragraph{Step 2. Existence of a maximizing schedule.}
We first show that $V_\Lambda$ is continuous on $[0,\infty)^n$.
For a fixed checkpoint set $T$, its cost is continuous wherever all selected stepsizes exceed $2$.
Near a finite schedule with a selected stepsize equal to $2$, its cost tends to $+\infty$ along admissible schedules, since the numerator factors stay positive while a denominator factor tends to zero.
Such a set cannot minimize the cost nearby, because the empty-set cost $c(\Lambda+\sum_t h_t)$ stays bounded.
A set with a selected stepsize below $2$ remains inadmissible in a neighborhood.
Thus, near any fixed schedule, $V_\Lambda$ is the minimum of the finitely many continuous costs of sets admissible at that schedule, proving continuity.

Choose schedules $H^{(k)}$ such that $V_\Lambda(H^{(k)})\to U_n(\Lambda)$.
We show that this sequence is bounded.
Otherwise, pass to a subsequence along which every coordinate has a finite or infinite limit and at least one tends to $+\infty$.
Let $j$ be the first coordinate tending to $+\infty$, so the preceding coordinates stay bounded.
For all sufficiently large $k$, splitting at $j$ as in Step~1 gives
\[
 V_\Lambda(H^{(k)})
 \le\frac{c\bigl(a+\sum_{t<j}h_t^{(k)}\bigr)}{h_j^{(k)}-2}
       U_{n-j}(\Lambda)
 \longrightarrow0.
\]
This contradicts $V_\Lambda(H^{(k)})\to U_n(\Lambda)\ge V_\Lambda(0,\ldots,0)=c\Lambda>0$.
Hence a subsequence converges to some $\bar H\in[0,\infty)^n$, and continuity gives $V_\Lambda(\bar H)=U_n(\Lambda)$.

\paragraph{Step 3. Tightness of the empty set.}
Fix a maximizing schedule $H$, and write $V:=V_\Lambda(H)=U_n(\Lambda)$.
The remaining argument uses the following observation.
If a perturbation strictly increased the cost of every tight set, all other costs would remain above $V$ for sufficiently small perturbations, by the continuity and divergence arguments in Step~2.
Since there are finitely many checkpoint sets, their new minimum would then exceed $V=U_n(\Lambda)$, a contradiction.

Suppose that the empty set were not tight.
By \cref{global:lem-log-submodularity}, the intersection of all tight sets is tight, and hence nonempty.
Choose an index $j$ in this intersection and slightly decrease $h_j$, keeping $h_j>2$.
Every tight set selects $j$, so its cost depends on $h_j$ only through the factor $1/(h_j-2)$ and strictly increases.
The observation above gives a contradiction, proving that the empty set is tight.

\paragraph{Step 4. Existence of a tight singleton.}
There must also be a nonempty tight set.
Otherwise, increasing any step slightly would increase the only tight cost, that of the empty set, contradicting the observation in Step~3.

Following~\citep[Appendix~C.6]{JungChoYun2026}, let $T_*$ be the intersection of all nonempty tight sets.
By \cref{global:lem-log-submodularity}, the intersection of two nonempty tight sets is nonempty and tight, so the same holds for $T_*$.
Suppose that $T_*$ contains distinct indices $i,j$, and put $u=h_i-2$ and $v=h_j-2$, with $u\le v$.
If $u<v$, choose $1<\theta<v/u$ and replace $(u,v)$ by $(u-t,v+\theta t)$ for sufficiently small $t>0$.
If $u=v$, replace them by $(u-t,u+t+t^2/(2u))$.
In both cases, the sum increases and the product decreases.
The empty-set cost therefore increases.
Every nonempty tight set selects both $i$ and $j$, so its cost depends on these stepsizes only through $1/(uv)$ and also increases.
This contradicts the observation in Step~3.
Thus $T_*$ is a tight singleton.

\end{proof}

\subsection{Proof of Lemma~\ref{thm:weighted-recursion}}
\label{global:app-uniform-growth}
\label{global:app-extremizers}

We first prove the following scalar inequality, which will be used in the inductive proof of \cref{thm:weighted-recursion}.

\begin{lemma}
\label{global:lem-scalar-condition}
Let $1<c\le2$, $p\in[\psil+c-1,3/2]$, and $\nu=1/p$.
If $x,y\ge a$ and $B>0$ satisfy \eqref{eq:empty-singleton}, i.e.
\[
 W=x+y+B-(a-2),\qquad WB=cxy,
\]
then $W^\nu\le x^\nu+y^\nu$.
\end{lemma}

\begin{proof}
Set $\overline W:=(x^\nu+y^\nu)^{1/\nu}$. Our goal is to show $W\le\overline W$.
Define $\psi(X):=X^2-(x+y)X-cxy$. Eliminating $B$ gives
\[
 \psi(W)=-(a-2)W\le0.
\]
Since $\psi$ has exactly one positive root and $W,\overline W>0$, it suffices to prove $\psi(\overline W)\ge0$.

Set $z=(x/\overline W)^\nu$, so that $1-z=(y/\overline W)^\nu$ and $z\in(0,1)$.
Since $p=1/\nu$, we have
\[
 \frac{\psi(\overline W)}{\overline W^2}=1-f(z),\qquad
 f(z):=z^p+(1-z)^p+c[z(1-z)]^p.
\]
Thus it remains to show $f(z)\le1$ on $[0,1]$.
Clearly $f(0)=1$. Our choice of $p$ gives
\[
 c\le1+p-\psil\le2^{2p}-2^{p+1},
\]
where the second inequality holds because the rightmost expression equals
$1$ at $p=\psil$ and has derivative
$2^{p+1}(\ln2)(2^p-1)>1$ for $p>1$.
Consequently, $f(1/2)=2^{1-p}+c2^{-2p}\le1$.

On $(0,1/2)$, the derivative is
\[
 f'(z)=p[z(1-z)]^{p-1}g(z),\qquad
 g(z):=(1-z)^{1-p}-z^{1-p}+c(1-2z).
\]
Moreover,
\[
 g''(z)=p(p-1)\bigl[(1-z)^{-p-1}-z^{-p-1}\bigr]<0,
 \qquad g(0+)=-\infty,\qquad g(1/2)=0.
\]
Strict concavity shows that $g$ either stays negative on $(0,1/2)$ or
changes sign once from negative to positive.
Thus $f$ either decreases throughout or first decreases and then increases,
so its maximum on $[0,1/2]$ occurs at an endpoint.
By symmetry, $f(z)\le1$ for every $z\in[0,1]$.
Hence $\psi(\overline W)\ge0$, so $W\le\overline W$, as required.
\end{proof}

\begin{proof}[Proof of \cref{thm:weighted-recursion}]
We induct on $n$, simultaneously for all $\Lambda\ge a$.
The case $n=0$ is $U_0(\Lambda)=c\Lambda$.
For $n\ge1$, choose a maximizing schedule $H$. 
By \cref{global:lem-extremal-schedule}, the empty set and a singleton $\{t_*\}$ are tight.
Put
\[
 H_L:=H_{1:t_*-1},\qquad H_R:=H_{t_*+1:n},\qquad
 W:=U_n(\Lambda)/c,
\]
and
\[
 x:=a+\sum_{t<t_*}h_t,\qquad
 y:=\Lambda+\sum_{t>t_*}h_t,\qquad B:=h_{t_*}-2.
\]
Then $x,y\ge a$ and $B>0$. Tightness of $\varnothing$ and $\{t_*\}$ gives
\[
 W=x+y+B-(a-2),\qquad WB=cxy,
\]
so \cref{global:lem-scalar-condition} gives $W^\nu\le x^\nu+y^\nu$.
The recursive decomposition \eqref{global:eq-value-decomposition} and the empty-set costs on the two sides give
\[
 cW\le\frac{V_a(H_L)V_\Lambda(H_R)}B
 \le\frac{c^2xy}B=cW.
\]
Hence $V_a(H_L)=cx$ and $V_\Lambda(H_R)=cy$. Applying the induction hypothesis to the two shorter schedules gives
\[
 x^\nu\le a^\nu(t_*-1)+a^\nu,\qquad
 y^\nu\le a^\nu(n-t_*)+\Lambda^\nu.
\]
Combining these bounds proves $W^\nu\le a^\nu n+\Lambda^\nu$, as required.
\end{proof}

\subsection[Proof of (\ref{eq:parameter-growth})]{Proof of \eqref{eq:parameter-growth}}
\label{global:app-parameters}
\label{global:app-nonanytime}

Fix $0<\eta\le1/10$ and use the parameter choice
\eqref{global:eq-parameter-choice}.  Then
$p\in[\psil,3/2]$ and $0<\epsilon<1/2$.  Moreover,
$\epsilon^2\le(\eta/16)^2\le\eta/6$, so
\[
 c^2=1+\frac\eta2+\epsilon^2
 \le1+\frac{2\eta}{3}
 \le\left(1+\frac\eta3\right)^2.
\]
Consequently $1<c<2$ and
\[
 c-1\le\frac\eta3<\eta=p-\psil.
\]
Thus \cref{thm:weighted-recursion} applies.

The definition of $a$ in \cref{thm:transfer} can be written as
\[
 \begin{gathered}
 R:=\sqrt{\epsilon^2+(1+\gamma)^2},\qquad
 H_0:=\left[
   \frac{c(1+\gamma+R)}{\epsilon(\gamma+R)}
 \right]^{R/\gamma}-1,\\
 C_0:=\epsilon+\frac{H_0+1}{\epsilon},
 \qquad
 a:=\max\{2,2C_0/c\}.
 \end{gathered}
\]
Under \eqref{global:eq-parameter-choice}, $c$ and $R$ are bounded above and
bounded away from zero by absolute constants, while
\[
 \log(1/\epsilon)=p\log(16/\eta)=O(\log(1/\eta)),
 \qquad
 \frac{R}{\gamma}=O(1/\eta).
\]
Since $\log\!\left[\frac{c(1+\gamma+R)}{\epsilon(\gamma+R)}\right]=\log(1/\epsilon)+O(1)=O(\log(1/\eta))$, we obtain
\[
 \log(1+H_0)
 =\frac{R}{\gamma}
  \log\!\left(
   \frac{c(1+\gamma+R)}{\epsilon(\gamma+R)}
  \right)
 =O\!\left(\eta^{-1}\log(1/\eta)\right).
\]
Since $C_0=\epsilon+(H_0+1)/\epsilon$, the stated choice of $a$ proves
\eqref{eq:parameter-growth} with an absolute implied constant.

\section{Proofs for the Anytime Lower Bound}
\label{app:anytime}

To prove \cref{thm:anytime}, we combine two consequences of \cref{thm:transfer}.
Selecting no checkpoints gives the one-dimensional Huber bound, which forces the sum of stepsizes to be large when the error is small.
Selecting only the last step as a checkpoint, when its stepsize exceeds $2$, gives a second lower bound in terms of that stepsize and the preceding sum.
Following~\citep[Section~4.1]{TsaiFatkhullinZhangHe2026} and~\citep[Section~4]{JungChoYun2026}, we apply these bounds at iterations when the stepsize is the largest so far.

To make this argument precise, we first introduce some notation.
Fix an infinite nonnegative schedule $H=(h_t)_{t\ge1}$.
For each prefix $H_n=(h_1,\ldots,h_n)$, denote its total stepsize, largest stepsize, and worst-case error by
\[
    S_n:=\sum_{t=1}^n h_t,
    \qquad M_n:=\max_{1\le t\le n}h_t,
    \qquad r_n:=\mathcal R_n(H_n).
\]
We call $n$ a \emph{record time} if $h_n=M_n$, and a \emph{strict record time} if $n\ge2$ and $h_n>M_{n-1}$.

For $p\in(\psil,\psil+1/10]$, set $\nu:=1/p\in(0,1)$ and choose $a,c$ as in \cref{global:app-parameters}, with $\eta=p-\psil$.
The two lower bounds above still depend on the schedule through $S_n$ and $M_n$.
To obtain a bound in terms of $n$, the next lemma uses \eqref{eq:cost-growth} to bound $S_n$ in terms of $n$ and $M_n$ when $r_n$ is small.

\begin{lemma}[Total stepsize at a record time]
\label{any:lem:record-sum}
For every $p\in(\psil,\psil+1/10]$, there are constants $C_p,r_{0,p}>0$ such that every record time $n$ with $M_n\ge4$ and $r_n\le r_{0,p}$ satisfies
\begin{equation}
    S_n\le C_p nM_n^{1-\nu}.
    \label{any:eq:record-sum}
\end{equation}
\end{lemma}

\subsection{Proof of Lemma~\ref{any:lem:record-sum}}
\label{any:app-record-sum}

Take $r_{0,p}:=1/8$ and fix a record time $n$ satisfying the lemma's assumptions.
Write
\[
    M:=M_n=h_n,
    \qquad H^-:=H_{n-1},
    \qquad V:=V_a(H^-).
\]

\paragraph{The prefix cost.}
We first use the last checkpoint to bound the prefix cost $V$ from below.
Take an admissible checkpoint set $T$ attaining $V_a(H^-)$ and append the final checkpoint $n$.
The factorization in the proof of \cref{global:lem-value-decomposition} gives
\[
    Q_a(T\cup\{n\};H_n)=\frac{caV}{M-2}.
\]
Since $n$ is the last step, the final gap is empty.
By \eqref{global:eq-checkpoint-expression}, the product in \eqref{eq:transfer} equals $[(M-2)/V]^2$.
Hence \cref{thm:transfer} gives
\begin{equation}
    r_n\ge\frac{(M-2)^2}{4V^2},
    \qquad
    V\ge\frac{M-2}{2\sqrt{r_n}}.
    \label{any:eq:prefix-cost-lower}
\end{equation}

\paragraph{Counting large steps.}
To bound the sum of stepsizes in $H^-$, we count how many exceed each level.
For $u\ge0$, define
\[
    N_{H^-}(u):=\bigl|\{t<n:h_t>u\}\bigr|.
\]
Now fix $u\in(2,M)$ and set $k:=N_{H^-}(u)$.
Use all stepsize indices $t<n$ with $h_t>u$ as checkpoints; they are admissible because $u>2$.
Write $b_1,\ldots,b_k$ for their stepsizes in time order and let $H^{(0)},\ldots,H^{(k)}$ be the blocks between them, including the initial and final blocks.
Write $m_j$ for the length of $H^{(j)}$.
Since the prefix has $n-1$ steps,
\begin{equation}
    \sum_{j=0}^k(m_j+1)=n.
    \label{any:eq:block-length-sum}
\end{equation}

Combine these checkpoints with admissible checkpoint sets attaining $V_a(H^{(j)})$ in all $k+1$ blocks.
The decomposition in \cref{global:lem-value-decomposition} gives
\begin{equation}
    V_a(H^-)
    \le V_a(H^{(k)})
    \prod_{i=1}^k\frac{V_a(H^{(i-1)})}{b_i-2}.
    \label{any:eq:block-factorization}
\end{equation}
Since $V_a(H^{(j)})\le U_{m_j}(a)$, \eqref{eq:cost-growth} gives
\[
    V_a(H^{(j)})^\nu\le(ca)^\nu(m_j+1)
    \qquad(0\le j\le k).
\]
Taking the $\nu$th power of \eqref{any:eq:block-factorization}, using $b_i-2>u-2$, and then applying AM--GM with \eqref{any:eq:block-length-sum}, we obtain
\begin{align}
    V^\nu
    &\le(ca)^\nu
      \left(\frac{(ca)^\nu}{(u-2)^\nu}\right)^k
      \prod_{j=0}^k(m_j+1)\notag\\
    &\le(ca)^\nu
      \left(\frac{(ca)^\nu}{(u-2)^\nu}\right)^k
      \left(\frac{n}{k+1}\right)^{k+1}.
    \label{any:eq:block-amgm}
\end{align}
Since $M-2\ge u-2$, comparison of \eqref{any:eq:prefix-cost-lower} and \eqref{any:eq:block-amgm} gives
\[
    (2\sqrt{r_n})^{-\nu}
    \le\left[\frac{n}{k+1}\left(\frac{ca}{u-2}\right)^\nu\right]^{k+1}.
\]
Since $r_n\le1/8$, the left-hand side exceeds one.
The quantity in brackets must therefore exceed one, giving
\begin{equation}
    N_{H^-}(u)+1
    \le n\left(\frac{ca}{u-2}\right)^\nu
    \qquad(2<u<M).
    \label{any:eq:excess-step-count}
\end{equation}
At a record time, every stepsize in $H^-$ is at most $M$, so integrating $N_{H^-}$ over $[0,M]$ recovers their sum.
Using \eqref{any:eq:excess-step-count} for $2<u<M$ and $N_{H^-}(u)\le n$ for $0\le u\le2$, we obtain
\begin{align}
    \sum_{t=1}^{n-1}h_t
    &=\int_0^M N_{H^-}(u)\,du\notag\\
    &\le2n+n(ca)^\nu\int_2^M(u-2)^{-\nu}\,du\notag\\
    &\le\left(2+\frac{(ca)^\nu}{1-\nu}\right)nM^{1-\nu}.
    \label{any:eq:prefix-sum}
\end{align}

\paragraph{Adding the last step.}
It remains to bound the final stepsize $h_n=M$.
Applying \eqref{eq:cost-growth} to the $n-1$ steps of $H^-$ and using \eqref{any:eq:prefix-cost-lower} gives
\[
    \frac{(M-2)^\nu}{(2\sqrt{r_n})^\nu}
    \le V^\nu
    \le(ca)^\nu n.
\]
Since $M\ge4$ and $r_n\le1/8<1/4$, it follows that
\[
    M^\nu\le(2ca)^\nu n,
    \qquad
    M\le(2ca)^\nu nM^{1-\nu}.
\]
Adding this to \eqref{any:eq:prefix-sum} proves the lemma with
\begin{equation}
    C_p:=2+\frac{(ca)^\nu}{1-\nu}+(2ca)^\nu.
    \label{any:eq:Cp-choice}
\end{equation}
\hfill$\square$

\subsection{Proof of Theorem~\ref{thm:anytime}}
\label{any:app-main-proof}
Selecting no checkpoints in \cref{thm:transfer} gives
\begin{equation}
    r_n\ge\frac{1}{4(1+S_n)}.
    \label{any:eq:total-sum-lower}
\end{equation}
If $H$ is bounded, then $S_n=O(n)$ and \eqref{any:eq:total-sum-lower} proves the theorem.
Suppose now that $H$ is unbounded, and fix a strict record time $n$ with $M_n=h_n\ge4$.
There are infinitely many such times.

At each such record time, applying \cref{thm:transfer} with the admissible checkpoint set $\{n\}$ gives
\begin{equation}
    r_n\ge\frac14
    \left[\frac{M_n-2}{c(a+S_{n-1})}\right]^2.
    \label{any:eq:final-step-lower}
\end{equation}
Since $M_n-2\ge M_n/2$ and $S_n\ge M_n\ge4$, we obtain
\[
    M_n\le4c(a+S_{n-1})\sqrt{r_n}
    \le c(a+4)S_n\sqrt{r_n}.
\]
If $r_n\le1/8$, substituting this bound on $M_n$ into \cref{any:lem:record-sum} gives
\[
    S_n^\nu\le C_p[c(a+4)]^{1-\nu}n r_n^{(1-\nu)/2}.
\]
The Huber bound \eqref{any:eq:total-sum-lower} also gives $S_n\ge1/(8r_n)$, so
\[
    1\le8^\nu C_p[c(a+4)]^{1-\nu}n r_n^{(1+\nu)/2}.
\]
Consequently, at every strict record time with $M_n\ge4$ and $r_n\le1/8$,
\begin{equation}
    r_n\ge c_p n^{-2p/(1+p)},
    \qquad c_p:=\left[8^\nu C_p[c(a+4)]^{1-\nu}\right]^{-2/(1+\nu)}.
    \label{any:eq:fixed-p-rate}
\end{equation}

To let $p$ approach $\psil$, we need to control the dependence of $c_p$ on $p$.
Recall that $\eta=p-\psil$.
As $\eta\downarrow0$, the quantities $\nu$, $1-\nu$, and $c$ remain bounded above and away from zero.
The choices \eqref{any:eq:Cp-choice} and \eqref{any:eq:fixed-p-rate}, together with \eqref{eq:parameter-growth}, give
\begin{equation}
\begin{aligned}
    \log C_p,\ \log(1/c_p)&=O\!\left(\eta^{-1}\log(1/\eta)\right),\\
    \frac{2p}{1+p}&=\pany+O(\eta).
\end{aligned}
\label{any:eq:parameter-growth}
\end{equation}
At each sufficiently large strict record time, choose
\[
    \eta=\sqrt{\frac{\log\log n}{\log n}}.
\]
If $r_n>r_{0,p}=1/8$, the claimed bound holds for sufficiently large $n$.
Otherwise, \eqref{any:eq:fixed-p-rate} and \eqref{any:eq:parameter-growth} give
\[
    r_n\ge n^{-\pany}
    \exp\!\left\{-C\left(\eta\log n+
    \eta^{-1}\log(1/\eta)\right)\right\}
    \ge n^{-\left(\pany+C'\sqrt{\log\log n/\log n}\right)}.
\]
There are infinitely many strict record times, which proves \cref{thm:anytime}.

\addtocontents{toc}{\protect\setcounter{tocdepth}{3}}
\end{document}